\documentclass[11pt, draftcls, onecolumn]{IEEEtran}

\makeatletter
\def\ps@headings{%
\def\@oddhead{\mbox{}\scriptsize\rightmark \hfil \thepage}%
\def\@evenhead{\scriptsize\thepage \hfil \leftmark\mbox{}}%
\def\@oddfoot{}%
\def\@evenfoot{}}
\makeatother
\usepackage{hyperref}

\usepackage{multirow}
\usepackage{xcolor}
\usepackage{graphicx,epstopdf}
\usepackage{amsmath,amsthm,amssymb,amsfonts}
\usepackage{tabularx,pdfcomment}
\usepackage{mathdots}

\usepackage{stmaryrd}

\usepackage{subfigure}

\usepackage{url}
\usepackage{cases}
\usepackage{bbm}
\usepackage{bm}

\usepackage{algorithmic}
\usepackage{algorithm}

\usepackage[square, comma, sort&compress, numbers]{natbib}

\theoremstyle{plain} \newtheorem{theorem}{Theorem}
\theoremstyle{plain} \newtheorem{definition}{Definition}
\theoremstyle{plain} \newtheorem{lemma}{Lemma}
\theoremstyle{plain} \newtheorem{remark}{Remark}
\theoremstyle{plain} 
\theoremstyle{plain} 
\theoremstyle{plain} 
\theoremstyle{plain} \newtheorem{example}{Example}

\DeclareMathOperator*{\argmin}{arg\,min}

\begin{document}

\title{Fourth-order Tensors with Multidimensional Discrete Transforms}

\author{Xiao-Yang Liu and Xiaodong Wang
\IEEEcompsocitemizethanks{\IEEEcompsocthanksitem  X.~Liu and X.~Wang are with the Department of Electrical Engineering, Columbia University.
\IEEEcompsocthanksitem  X.~Liu is with the Department of Computer Science and Engineering, Shanghai Jiao Tong University.
}
\thanks{}}

\maketitle

\IEEEdisplaynotcompsoctitleabstractindextext
\IEEEpeerreviewmaketitle

\begin{abstract}

    The big data era is swamping areas including data analysis, machine/deep learning, signal processing, statistics, scientific computing, and cloud computing. The multidimensional feature and huge volume of big data put urgent requirements to the development of multilinear modeling tools and efficient algorithms. In this paper, we build a novel multilinear tensor space that supports useful algorithms such as SVD and QR, while generalizing the matrix space to fourth-order tensors was believed to be challenging. Specifically, given any multidimensional discrete transform, we show that fourth-order tensors are bilinear operators on a space of matrices. First, we take a transform-based approach to construct a new tensor space by defining a new multiplication operation and tensor products, and accordingly the analogous concepts: identity, inverse, transpose, linear combinations, and orthogonality. Secondly, we define the $\mathcal{L}$-SVD for fourth-order tensors and present an efficient algorithm, where the tensor case requires a stronger condition for unique decomposition than the matrix case. Thirdly, we define the tensor $\mathcal{L}$-QR decomposition and propose a Householder QR algorithm to avoid the catastrophic cancellation problem associated with the conventional Gram-Schmidt process. Finally, we validate our schemes on video compression and one-shot face recognition. For video compression, compared with the existing tSVD, the proposed $\mathcal{L}$-SVD achieves $3\sim 10$dB gains in RSE,  while the running time is reduced by about $50\%$ and $87.5\%$, respectively. For one-shot face recognition, the recognition rate is increased by about $10\% \sim 20\%$.

\end{abstract}



\section{Introduction}\label{sect:introduction}

  Driving by a rapidly growing number of sensor devices and sensing systems of rapidly growing resolution, the big data era is swamping areas including data analysis, machine/deep learning \cite{lecun2015deep}, signal processing, statistics, scientific computing, and cloud computing. The exponential explosion of big data featured as multidimensional and huge volume has highlighted the limitations of standard flat-view matrix models and the necessity to move toward more versatile data analysis tools \cite{baraniuk2011more}. Thus, the successful problem-solving tools provided by the numerical linear algebra need to be broadened and generalized. Higher-order tensor\footnote{Also known as multiway, $n$-way, multidimensional array, or multilinear data array in the literature.} modeling \cite{kolda2009tensor} together with efficient tensor-based algorithms enables such a fundamental paradigm shift.

  Tensors, as multilinear modeling tools, have attracted tremendous interests in recent years. The major advantages of representing data arrays as tensor models over matrx/vectors models are as follows: 1) tensor decompositions guarantee uniqueness which is useful for blind source separation \cite{cichocki2009nonnegative, comon2014tensors}, etc; 2) tensor modeling including low-rank tensor decomposition \cite{kolda2009tensor,papalexakis2016tensors,oseledets2011tensor,cichocki2015tensor} and tensor networks \cite{cichocki2016tensor,Chichochi2013}, turns the curse of dimensionality into a blessing of dimensionality \cite{cichocki2015tensor} by decomposing the big data arrays into much smaller latent factors,
  and 3) data analysis techniques using tensor models have great flexibility in the choice of constraints that match data properties, thus extract more meaningful latent components than matrix-based methods. Examplar applications include multiway component analysis \cite{cichocki2015tensor,lu2008mpca}, blind source separation \cite{cichocki2009nonnegative, comon2014tensors}, dimensionality reduction \cite{sidiropoulos2014parallel}, feature extraction \cite{zhang2013tensor}, classification/clustering and pattern recognition \cite{sun2016heterogeneous}, topic modeling \cite{anandkumar2014tensor}, and deep neural networks \cite{novikov2015tensorizing,janzamin2015beating, cohen2015expressive}.

  Existing tensor models \cite{kolda2009tensor,papalexakis2016tensors,oseledets2011tensor,cichocki2015tensor} treat tensors as multidimensional arrays of real values upon which algebraic operations generalizing matrix operations can be performed. The corresponding tensor spaces are viewed as the tensor product (namely the Kronecker product) of vector spaces, including 1) the canonical polyadic (CP) decomposition, the Tucker decomposition, and their variants \cite{kolda2009tensor}; 2) the higher-order SVD (HOSVD) \cite{de2000multilinear}; 3) the recently proposed tensor-train decomposition \cite{oseledets2011tensor} that is further extended to the tensor ring decomposition \cite{zhao2016tensor}; and 4) tensor networks \cite{cichocki2016tensor,Chichochi2013}. All the above tensor models rely on either contraction products (e.g., the $n$-mode product \cite{kolda2009tensor}, the vector inner product, and the Einstein product \cite{Einstein2007book}) or expansion products (e.g., the vector outer product and the Kronecker product). However, those two kinds of products will change tensors' order. Therefore, the corresponding tensor spaces lack the closure property, being fundamentally different from the well-studied conventional matrix space, thus the classic algorithms (SVD and QR, etc.) do not hold.

\subsection{Circular unfolding-folding based low-tubal-rank tensor model}
\label{subsec:circular_unfolding_folding_low_tubal_rank_tensor_model}

  The low-tubal-rank tensor model \cite{kilmer2011factorization,kilmer2013third} is the first trial to extend the conventional matrix space to third-order tensors \cite{braman2010third}. It is based on a circular unfolding-folding scheme (formally presented in Definition \ref{def:t_product}) that introduces structured redundancy by the circular unfolding process. The authors defined the t-product between two third-order tensors as follows: first unfold the left tensor into a block circulant matrice and the right tensor into a tall matrix, then perform conventional matrix multiplication between those two matrices, and finally fold the result matrix back into a third-order tensor. Under this new algebraic framework, \cite{kilmer2011factorization,kilmer2013third} generalized all classical algorithms of the conventional matrix space, such as SVD, QR, normalization, the Gram-Schmidt process, power iteration, and Krylov subspace methods. Further, it is shown  \cite{martin2013order} that this circular unfolding-folding scheme can be used to recursively define the tensor SVD decomposition for higher-order tensors.

  The authors \cite{kilmer2011factorization,kilmer2013third} pointed out that the t-product $*_t$ between two $1 \times 1 \times n$ tensors is in fact equivalent to the discrete circular convolution $\odot$ of two vectors, i.e., $\bm{x},\bm{y} \in \mathbb{R}^{1 \times 1 \times n}$, $\bm{x} *_t \bm{y} = \text{circ}(\bm{x}) \bm{y} = \bm{x} \odot \bm{y}$ where $\text{circ}(\bm{x})$ is the circular matrix derived from $\bm{x}$ (formally given in (\ref{eq:circ_vector})). Thus, the t-product between two third-order tensors is analogous to the traditional matrix multiplication between two matrices whose entries are $1 \times 1 \times n$ tensors, where the conventional scalar product is replaced by the discrete circular convolution. Moreover, from a computational perspective, it is shown \cite{kilmer2011factorization,kilmer2013third} that the t-product can be computed efficiently by performing a discrete Fourier transform (using FFT) along the tubal fibers of each third-order tensor, performing pair-wise matrix products for all frontal slices of the two tensors in the ``transform domain" (i.e. frequency domain), and then applying an inverse DFT along the tubal fibers of the result tensor. Therefore, the circular unfolding-folding based scheme defines operations based on the Fourier transform, as pointed out in Remark \ref{remark:t_product_fourier_transform}.

  Recently, this new tensor model is successfully applied to many engineering areas, such as seismic data processing \cite{Shuchin2015Seimic} and data completion \cite{zhang2016exact}, WiFi fingerprint-based indoor localization \cite{liu2016adaptive}, drone-based wireless relay \cite{xie2016drone}, MRI imaging \cite{Misha2014CT}, video compression and denoising \cite{zhang2014novel,liu2016low}, image clustering \cite{kernfeld2014clustering}, two-dimensional dictionary learning \cite{jiang2017graph}, and face recognition \cite{hao2013facial}. The authors in \cite{liu2016low} pointed out that, compared with other tensor models, the low-tubal-rank tensor model is superior in capturing  a ``spatial-shifting" correlation that is ubiquitous in real-world data arrays.

\subsection{Motivation to fourth-order tensor model with multidimensional discrete transforms}

  We are motivated to propose a new tensor model with general discrete transform due to the following observations:
  \begin{itemize}
   \item The t-product \cite{kilmer2011factorization,kilmer2013third} has a major disadvantage in that for real tensors, the FFT-based implementations of the t-product and the t-SVD factorization \cite{kilmer2011factorization,kilmer2013third} require intermediate complex arithmetic. Even taking advantage of complex symmetry in the Fourier domain, the complex arithmetic is much more expensive than real arithmetic. Therefore, we are interested in tensor models that involves only real-valued fast transforms, whose algorithms are faster than their counterparts in the low-tubal-rank tensor model \cite{kilmer2011factorization,kilmer2013third}.
    \item Combining the low-rank property and the transform-domain sparsity (not limiting to the frequency domain \cite{candes2006robust,lustig2008compressed}), we expect new tensor models to possess improvements in terms of compression ratios \cite{sidiropoulos2012multi} and accordingly better compression ratios lead to the design of faster algorithms.
  \end{itemize}

  Real-world data arrays exhibit strong sparsity in various multidimensional discrete transform domains \cite{candes2006robust,lustig2008compressed} besides the frequency domain. First, in EEG (electroencephalography) and MEG (magnetoencephalography) imaging \cite{becker2015brain}, widely used assumptions are: minimum energy in a transform domain, sparsity in a Fourier domain that is modeled as a space-time-frequency tensor, separability in space and wave-vector domain (for the spatial distribution of the sources), and separability in space and frequency domain (for the temporal distribution of the sources) that is modeled as a space-time-wave-vector tensor. Secondly, MRI \cite{lustig2008compressed} is naturally \textit{compressible} by sparse coding in the wavelet transform, and MRI scanners naturally acquires \textit{spatial-frequency encoded} samples rather than direct pixel samples, such as the single-slice 2DFT, multislice 2DFT, and 3DFT imaging, while CT data is collected in the 2D Frequency domain \cite{candes2006robust,Misha2014CT}. Thirdly, for image compression, JPEG utilizes the discrete cosine transform (DCT) \cite{wallace1992jpeg} and JPEG-$2000$ utilizes the wavelet transform. Fourthly, in face recognition \cite{he2005face}, the rotation and lighting effects can be captured by transform operations, while images can be treated as in the same class \cite{he2011laplacian} if there only differs in terms of rotation and distortion. Finally, in Internet of Things, sensory data can be represented as tensors of time series \cite{lim2010multiarray} that are periodic.

  We focus on the fourth-order tensors with the following considerations:
    \begin{itemize}
    \item The circular unfolding-folding based scheme to define the fourth-order tensor \cite{martin2013order} is recursive and thus complicated, and such a definition scheme does apply to transforms without structured matrix expressions, e.g., the wavelet transforms \cite{daubechies1992ten}.
    \item Fourth-order tensors is ubiquitous in machine/deep learning tasks, e.g., the one-shot fact recognition problem in Section \ref{sec:performance_evaluation}. To learn meaningful and generalizable models, allowing abstract algebraic structures with corresponding manipulation operations is an attractive paradigm. For example in face recognition, faces are essentially the combination (with scaling, distortion, and rotation) of complex elementary structures and patterns \cite{he2005face,he2011laplacian} while existing tensor models \cite{kolda2009tensor,papalexakis2016tensors,cichocki2015tensor} simply treat face images as multilinear data arrays that cannot model the rotation effect.
    \item In sensory data recovery, existing tensor models \cite{kolda2009tensor,papalexakis2016tensors,oseledets2011tensor,cichocki2015tensor}  become invalid for various data loss patterns and are insufficient to allow versatile sampling schemes. Fourth-order tensors will allow losing/sampling slices, similar to the fact \cite{liu2016adaptive} that the low-tubal-rank third-order tensor model allows losing a time series or sampling vectors.
    \end{itemize}

  To generalize all classical algorithms for matrices to fourth-order tensors with general multidimensional discrete transforms, we encounter the following challenges:
  \begin{itemize}
    \item Although tensors are multilinear data arrays, existing tensor models \cite{kolda2009tensor,papalexakis2016tensors,oseledets2011tensor,cichocki2015tensor}  cannot be treated as ``multilinear operators".  A fundamental fact in linear algebra states that one can view the matrix-vector product $A\bm{x}$ by interpreting it as a weighted sum (linear combination) of the columns of $A$. This observation in matrix case does not hold for existing higher-order tensor spaces \cite{kolda2009tensor,papalexakis2016tensors,cichocki2015tensor}, therefore, the classic algorithms become invalid.
    \item The important matrix SVD process (eigendecomposition together with all useful processes) does not hold for high-order tensor models. This failure roots in the fact that \textit{the contraction products (used in existing tensor models) lack the closure property for odd-order tensors}.
  \end{itemize}

\subsection{Our contributions}

   In this paper, we build a novel multilinear tensor space that supports useful algorithms in the conventional matrix space, such as SVD and QR. Specifically, given any multidimensional discrete transform, one is able to construct a new tensor space, and then we can treat fourth-order tensors are bilinear operators on a space of matrices. Note that in previous works \cite{kolda2009tensor,papalexakis2016tensors,oseledets2011tensor,cichocki2015tensor,kilmer2011factorization,kilmer2013third,martin2013order},  generalizing the matrix space to fourth-order tensors was believed to be challenging.

%

   First, we take a transform-based approach to define a new multiplication operation and tensor products, and accordingly the analogous concepts: identity, inverse, transpose, linear combinations, and orthogonality. Specifying the discrete transform of interest to be a discrete Fourier transform and considering the third-order case, our results can recovered all results in the low-tubal-rank tensor model \cite{kilmer2011factorization,kilmer2013third}.

   Secondly, we define the $\mathcal{L}$-SVD for fourth-order tensors and present an efficient algorithm. The fundamental difference between $\mathcal{L}$-SVD and conventional SVD lies in the inequivalence between the tensor-eigenvalue equation and the tensor-eigenvector equation, as pointed out in Remark \ref{remark:tensor_has_more_eigenvectors}. Therefore, the tensor case requires a stronger condition for unique decomposition than the matrix case.

   Thirdly, we define the tensor $\mathcal{L}$-QR decomposition and propose a Householder QR algorithm. In the low-tubal-rank tensor model \cite{kilmer2011factorization,kilmer2013third}, the authors directly adopted the conventional Gram-Schmidt process to compute the QR decomposition, while the conventional Gram-Schmidt process will encounter the catastrophic cancellation problem. The proposed Householder QR algorithm can avoid such a problem, while it cannot be extended from the matrix case.

   Finally, compared with the existing t-SVD, the proposed $\mathcal{L}$-SVD's performance gain is $3\sim6$dB for video compression, and the accuracy is increased about $15\%$ for one-shot face recognition, while the running time is reduced by about $50\%$ and $87.5\%$, respectively.

   Finally, we apply the new tensor model to two examplar applications: video compression and one-shot face recognition. We utilize the proposed $\mathcal{L}$-SVD to compress an NBA basketball video and a drone video of the Central Park in autumn.  Compared with the existing tSVD and SVD, $\mathcal{L}$-SVD  achieves $3\sim 10$dB gains in RSE while the running time is reduced by $xx\%$ and $xx\%$, respectively. For one-shot face recognition, we use the Weizmann face database and the recognition rate is increased by about $10\% \sim 20\%$.

   The remainder of the paper is organized as follows. Section II introduces the notations and several basic operations. Section III defines a new tensor space from a transform-based approach. Section IV and V present the $\mathcal{L}$-SVD and $\mathcal{L}$-QR decompositions, including the definitions, computing algorithms and correctness proofs. Section VI describes the performance evaluation, and Section VII concludes this work.



\section{Notations and Basic Operators}

  We first introduce the notations and preliminaries. Then, we describe several basic operators to manipulate the data arrays.

\subsection{Notations}

  The \textit{order} of a tensor is the number of modes, also known as ways or dimensions, e.g., third-order tensors and fourth-order tensors. Scalars are denoted by lowercase letters, e.g., $a$; vectors are denoted by boldface lowercase letters, e.g., $\bm{a}$; matrices\footnote{A matrix is a second-order tensor, a vector is a first-order tensor, and a scalar is a tensor of order zero.} are denoted by boldface capital letters, e.g., $\bm{A}$; and higher-order tensors are denoted by calligraphic letters, e.g., a fourth-order tensor $\mathcal{A} \in \mathbb{R}^{n_1 \times n_2 \times n_3 \times n_4}$ where $\mathbb{R}$ denotes the set of real numbers.
   The transpose of a vector or a matrix are denoted with a superscript $T$, e.g., $\bm{a}^T$, $\bm{A}^T$, while the Hermitian transpose (conjugate transpose) are denoted with a superscript $\mathrm{H}$, e.g., $\bm{a}^{\mathrm{H}}$,  $\bm{A}^{\mathrm{H}}$.

  The $i$th element of a vector $\bm{a}$ is $\bm{a}_i$, the $(i,j)$th element of a matrix $\bm{A}$ is $\bm{A}_{ij}$ or $\bm{A}(i,j)$, and similarly for higher-order tensors, e.g.,  $\mathcal{X}_{ijk}$, $\mathcal{X}_{ijk\ell}$ or  $\mathcal{X}(i,j,k)$, $\mathcal{X}(i,j,k,\ell)$. The $k$th element in a sequence is denoted by a superscript index, e.g., $\bm{a}^{k}$ denotes the $k$th vector in a sequence of vectors while $\bm{A}^{k}$ denotes the $k$th matrix in a sequence of matrices. Subarrays are formed when a subset of the indices is fixed, e.g., the rows and columns of a matrix. A colon is used to indicate all elements of a mode, e.g., the $j$th column of $\bm{A}$ is denoted by $\bm{A}_{:j}$, and the $i$th row of a matrix $\bm{A}$ is denoted by $\bm{A}_{i:}$, alternatively, $\bm{A}_j$ and $\bm{A}_i^T$. We use $[n]$ to denote the index set $\{1,2,...,n\}$, and given $n\geq i$, $[i:n]$ denotes the index set $\{i,i+1,...,n\}$. Let $\textbf{det}(\bm{A})$ denote the determinant of a square matrix $\bm{A} \in \mathbb{R}^{n \times n}$.

  Let $\text{vec}(\mathcal{A})$ denote the vector representation of $\mathcal{A}$ (the ordering of the elements is not important so long as it is consistent), and $\text{fold}(\cdot)$ is the inverse operator that transforms $\text{vec}(\mathcal{A})$ back to $\mathcal{A}$. Given a vector $\bm{x} \in \mathbb{R}^{n}$, the $\ell_2$-norm is $||\bm{x}||_2 = \sqrt{\sum\limits_{i \in [n]} \bm{x}_i^2}$, while in the high-order case (matrices and tensors) it becomes the Frobenius norm ($F$-norm) defined as follows
  \begin{equation}
  || \mathcal{A} ||_F = \sqrt{\text{vec}(\mathcal{A})^2} = \sqrt{\sum\limits_{i \in [n_1]} \sum\limits_{j\in[n_2]}\sum\limits_{k\in[n_3]}\sum\limits_{\ell \in [n_4]}\mathcal{A}_{ijk\ell}^2}.
  \end{equation}
  The spectrum norm of a matrix $\bm{A} \in \mathbb{R}^{n_1 \times n_2}$ is defined in terms of the $\ell_2$-norm of a vector
  \begin{equation}\label{eq:matrix_spectrum_norm}
  ||\bm{A}|| = \max\limits_{\bm{x} \in \mathbb{R}^{n_2 \times 1},~||\bm{x}||_2 = 1} ||\bm{A} \bm{x}||_2.
  \end{equation}

\subsection{Basic Operators}

  We define serveral basic operators that will facilitate our description and analysis in Section \ref{sec:4D_tensor_decomposition}. Intuitively, the operator $\text{MatView}(\cdot)$ extracts a sequence of matrices from a fourth-order tensor, while $\text{TenView}(\cdot)$ is the inverse operator. Given two sequences $\text{MatView}(A)$ and $\text{MatView}(B)$, we exploit the block diagonal representation to represent the parallel matrix multiplications.

  The $\text{MatView}(\cdot)$ operator forms a fourth-order tensor into a sequence of matrices. Formally, $\text{MatView}(\cdot)$ takes a tensor $\mathcal{A} \in \mathbb{R}^{n_1 \times n_2 \times n_3 \times n_4}$ and returns a sequence of $n_1 \times n_2$ matrices, as follows
  \begin{equation}\label{eq:matveiew}
  \begin{split}
  \text{MatView}(\mathcal{A}) &= \{\bm{A}^1,...,\bm{A}^p,...,\bm{A}^P\},~~P=n_3n_4,~~p \in [P],\\
  \bm{A}^p(i,j) &= \mathcal{A}(i,j,k,\ell),~p=(k-1)n_3 + \ell,~i \in [n_1],~j\in[n_2],~k \in [n_3],~\ell \in [n_4].
  \end{split}
  \end{equation}
  The operator that folds $\text{MatView}(\mathcal{A})$ back to tensor $\mathcal{A}$ is defined as follows
  \begin{equation}\label{eq:tenview}
  \text{TenView}(\text{MatView}(\mathcal{A})) = \mathcal{A}.
  \end{equation}

  Given two fourth-order tensors $\mathcal{A} \in \mathbb{R}^{n_1 \times n' \times n_3 \times n_4}$ and $\mathcal{B} \in \mathbb{R}^{n' \times n_2 \times n_3 \times n_4}$, the two sequences $\text{MatView}(A)$ and $\text{MatView}(B)$ are both of size $P = n_3n_4$. The $p$th matrices are $\bm{A}^p \in \mathbb{R}^{n_1 \times n'}$ and $\bm{B}^p \in \mathbb{R}^{n' \times n_2}$, and their multiplication is well-defined as $\bm{C}^p =\bm{A}^p \bm{B}^p \in \mathbb{R}^{n_1 \times n_2}$. One can represent $\text{MatView}(A)$ as a much bigger block diagonal matrix as follows
  \begin{equation}\label{eq:block_diagonal_representation}
  \text{blkdiag}(\text{MatView}(\mathcal{A})) = \left[
   \begin{array}{ccc}
    \bm{A}^{1} &  & \\
    & \ddots & \\
    & & \bm{A}^{P}\\
    \end{array}
    \right].
  \end{equation}
  Then, the elementwise matrix multiplication of two sequences can be represented as
  \begin{equation}\label{eq:blkdiagonal_presentation_multiplication_two_matrix_sequence}
  \text{blkdiag}(\text{MatView}(\mathcal{C})) = \text{blkdiag}(\text{MatView}(\mathcal{A})) \cdot \text{blkdiag}(\text{MatView}(\mathcal{B})),
  \end{equation}
  where the operation $\cdot$ denotes the conventional matrix multiplication. Note that (\ref{eq:blkdiagonal_presentation_multiplication_two_matrix_sequence}) compactly represents the following $P$ parallel matrix multiplications
  \begin{equation}\label{eq:parallel_multiplication_block_structure}
  \left[
   \begin{array}{ccc}
    \bm{C}^{1} &  & \\
    & \ddots & \\
    & & \bm{C}^{P}\\
    \end{array}
    \right] = \left[
   \begin{array}{ccc}
    \bm{A}^{1} &  & \\
    & \ddots & \\
    & & \bm{A}^{P}\\
    \end{array}
    \right] \cdot \left[
   \begin{array}{ccc}
   \bm{B}^{1} &  & \\
    & \ddots & \\
    & & \bm{B}^{P}\\
    \end{array}
    \right].
  \end{equation}


 For vector $\bm{x} \in \mathbb{R}^{n}$, the corresponding circular matrix is
  \begin{equation}\label{eq:circ_vector}
  \text{circ}(\bm{x}) = \left[
  \begin{array}{ccccc}
  \bm{x}_{(1)} & \bm{x}_{(n)} & \bm{A}_{(n - 1)} & \cdots & \bm{x}_{(2)} \\
  \bm{x}_{(2)} & \bm{x}_{(1)} & \bm{x}_{(n)}     & \cdots & \bm{x}_{(3)} \\
  \vdots   & \ddots   & \ddots       & \ddots & \vdots \\
  \bm{x}_{(n)} & \bm{x}_{(n - 1)} & \ddots   & \bm{x}_{(2)} & \bm{x}_{(1)}\\
  \end{array}
  \right],
  \end{equation}

 For a fourth-order tensor $\mathcal{A} \in \mathbb{R}^{n_1 \times n_2 \times n_3 \times n_4}$, we use the notation $\bm{A}_{(i)} \in \mathbb{R}^{n_1 \times n_2 \times n_3}$ to denote the third-order tensor created by holding the $4$th index of $\mathcal{A}$ fixed at $i$, $i \in [n_4]$. We create the following block circulant representation
  \begin{equation}\label{eq:bcirc}
  \text{bcirc}(\mathcal{A}) = \left[
  \begin{array}{ccccc}
  \bm{A}_{(1)} & \bm{A}_{(n_4)} & \bm{A}_{(n_4 - 1)} & \cdots & \bm{A}_{(2)} \\
  \bm{A}_{(2)} & \bm{A}_{(1)} & \bm{A}_{(n_4)}     & \cdots & \bm{A}_{(3)} \\
  \vdots   & \ddots   & \ddots       & \ddots & \vdots \\
  \bm{A}_{(n_4)} & \bm{A}_{(n_4 - 1)} & \ddots   & \bm{A}_{(2)} & \bm{A}_{(1)}\\
  \end{array}
  \right],
  \end{equation}
  where $\text{bcirc}(\mathcal{A}) \in \mathbb{R}^{n_1n_4 \times n_2n_4 \times n_3}$.
  The $\text{unfold}(\cdot)$ command takes an $n_1 \times n_2 \times n_3 \times n_4$ tensor and returns an $n_1n_4 \times n_2 \times n_3$ block tensor as follows
  \begin{equation}\label{eq:unfold}
  \text{unfold}(\mathcal{A}) = \left[
  \begin{array}{ccccc}
  \bm{A}_{(1)} \\
  \bm{A}_{(2)} \\
  \vdots   \\
  \bm{A}_{(n_4)} \\
  \end{array}
  \right].
  \end{equation}
  The operation that takes $\text{unfold}(\mathcal{A})$ back to tensor $\mathcal{A}$ is the $\text{fold}(\cdot)$ command:
  \begin{equation}\label{eq:fold}
  \text{fold}(\text{unfold}(\mathcal{A})) = \mathcal{A}.
  \end{equation}


\section{New Tensor Space}
\label{sec:new_tensor_space}

We build a new tensor space for fourth-order tensors. More specifically, we define a novel tensor-scalar multiplication, and accordingly define the multiplication of two tensors, identity, inverse, transpose, diagonality, orthongonality, and subspaces.

\subsection{A New Tensor Space}

   We build a new tensor space in which fourth-order tensors act as linear operators in a way similar to the conventional matrix space. More specifically, this new tensor space views $n_1 \times n_2 \times n_3 \times n_4$ fourth-order tensors on a space of $n_1 \times n_2$ matrices with entries in $\mathbb{R}^{1 \times 1 \times n_3 \times n_4}$.

  \begin{definition}\label{def:tensor_scalar}
  (\textbf{Tensor-scalar}) We call an element of the space $\mathbb{R}^{1 \times 1 \times n_3 \times n_4}$ as a tensor-scalar. The set of tensor-scalars are denoted by $\mathfrak{R}$.
  \end{definition}

  Let $\bm{1} \in \mathbb{R}^{1 \times 1 \times n_3 \times n_4}$ denote the tensor scalar with all entries equal to $1$, and $\bm{0}$ denote the zero tensor scalar (its dimension will be clear from the context). The addition $+$ and multiplication $\bullet$ are two fundamental operations in a space. In the space $\mathfrak{R}$ of tensor-scalars, we set the addition operation to be the element-wise addition, while the the multiplication operation defined in the following is based on a two-dimensional discrete transform.

  \begin{definition}\label{def:tensor_scalar_multiplication}
  (\textbf{Tensor-scalar multiplication}) Given an invertible two-dimensional discrete transform $\mathcal{L}:~\mathfrak{R} \rightarrow \mathfrak{R}$, the elementwise multiplication $*$, and $\alpha,~\beta \in \mathfrak{R}$, we define the tensor-scalar multiplication
  \begin{equation}\label{eq:tensor_scalar_multiplication}
  \alpha \bullet \beta \triangleq \mathcal{L}^{-1}(\mathcal{L}(\alpha) * \mathcal{L}(\beta)),
  \end{equation}
  where the multidimensional transform $\mathcal{L}:~\mathfrak{R} \rightarrow \mathfrak{R}$ and inverse transform $\mathcal{L}^{-1}:~\mathfrak{R} \rightarrow \mathfrak{R}$ together perform a forward or backward transform on each tensor-scalar.
  \end{definition}

  In the following, the one-to-one mappings $\mathcal{L}:~\mathfrak{R}^{n_1 \times n_2} \rightarrow \mathfrak{R}^{n_1 \times n_2}$ and $\mathcal{L}^{-1}:~\mathfrak{R}^{n_1 \times n_2} \rightarrow \mathfrak{R}^{n_1 \times n_2}$ represent the forward and backward transforms on each tensor-scalars of the $n_1 \times n_2$ matrices. We introduce the notation $\widetilde{\mathcal{A}}$ to denote the transform-domain representation of $\mathcal{A} \in \mathfrak{R}^{n_1 \times n_2}$ such that $\widetilde{\mathcal{A}} = \mathcal{L}(\mathcal{A})$ and $\mathcal{A} = \mathcal{L}^{-1}(\widetilde{\mathcal{A}})$.

    \begin{definition}\label{def:tensor_scalar_magnitude_ordering}
  (\textbf{Magnitude and ordering of tensor-scalars})
  The magnitudes of $\alpha \in \mathfrak{R}$ is denoted as $\text{abs}(\cdot)$, defined in an elementwise way as follows
  \begin{equation}\label{eq:magnitude_definition}
  \text{abs}(\alpha) = \mathcal{L}^{-1}(|\widetilde{\alpha}|),
  \end{equation}
  where $|\cdot|$ denotes the absolute values in an elementwise manner. Then, we introduce the ordering of tensor-scalars as follows
  \begin{equation}\label{eq:ordering_tensor_scalar}
  \alpha \succeq \beta,~~\text{if}~\alpha_p \geq \beta_p~~\text{for~all}~p\in[n_3n_4].
  \end{equation}
  \end{definition}

  \begin{definition}
  (\textbf{Sign of a tensor-scalar}) Given a tensor-scalar $\alpha \in \mathfrak{R}$, we denote its sign as $\angle \alpha$, defined as follows
  \begin{equation}\label{eq:angle_definition}
  \angle \alpha \bullet \text{abs}(\alpha) = \alpha,
  \end{equation}
  where $\text{abs}(\cdot)$ is given in Definition \ref{def:tensor_scalar_magnitude_ordering}.
  \end{definition}

  \begin{remark}
  Combining (\ref{eq:tensor_scalar_multiplication}), (\ref{eq:magnitude_definition}) and (\ref{eq:angle_definition}), one can compute the sign of a tensor-scalar $\alpha \in \mathfrak{R}$ as follows
  \begin{equation}
  \begin{split}
  \angle \alpha &= \mathcal{L}^{-1}\left(\frac{\widetilde{\alpha}_1}{\mathcal{L}(\text{abs}(\alpha))_1}, \frac{\widetilde{\alpha}_2}{\mathcal{L}(\text{abs}(\alpha))_2}, ..., \frac{\widetilde{\alpha}_{n_3n_3}}{\mathcal{L}(\text{abs}(\alpha))_{n_3n_3}} \right)\\
  &= \mathcal{L}^{-1}\left(\frac{\widetilde{\alpha}_1}{|\widetilde{\alpha}_1|}, \frac{\widetilde{\alpha}_2}{|\widetilde{\alpha}_2|}, ..., \frac{\widetilde{\alpha}_{n_3n_3}}{|\widetilde{\alpha}_{n_3n_3}|} \right),
  \end{split}
  \end{equation}
  where $\mathcal{L}(\text{abs}(\alpha))_p = |\widetilde{\alpha}_p|$ for $p \in [n_3n_4]$.
  \end{remark}

  \begin{definition}
  (\textbf{Square roots of a tensor-scalar}) For a tensor-scalar $\alpha \in \mathfrak{R}$, the square roots of $\alpha$ is defined as: $\sqrt{\alpha} \triangleq \mathcal{L}^{-1}(\sqrt{\widetilde{\alpha}})$, where $\sqrt{\widetilde{\alpha}}$ is computed in an elementwise manner. Note that $\sqrt{\alpha}$ can be complex-valued, and the set of square roots $\sqrt{\alpha}$ could be as large as $2^{n_3n_4}$.
  \end{definition}

  \begin{lemma}\label{lemma:multiplicative_unity}
  (\textbf{Multiplicative unity $e$}) Let $e = \mathcal{L}^{-1}(\bm{1})$ where $\bm{1} \in \mathbb{R}^{1 \times 1 \times n_3 \times n_4}$ denotes an $n_3 \times n_4$ matrix with all entries equal to $1$, then $e$ is the multiplicative unity for the tensor-scalar multiplication $\bullet$.
  \end{lemma}
  \begin{proof}
  To show that $e$ is the multiplicative unity, we prove that $\alpha \bullet e = e \bullet \alpha = \alpha$ for any $\alpha \in \mathfrak{R}$. Since $e = \mathcal{L}^{-1}(\bm{1})$, we have $\mathcal{L}(e) = \bm{1}$. For the elementwise multiplication $*$, we have $\bm{1} * \mathcal{L}(\alpha) = \mathcal{L}(\alpha) * \bm{1} = \mathcal{L}(\alpha)$. Since $\mathcal{L}$ is an invertible transform, namely a bijection mapping, we apply the inverse transform $\mathcal{L}^{-1}$ to both sides and get $\alpha \bullet e = e \bullet \alpha = \alpha$.
  \end{proof}

  \begin{lemma}\label{lemma:multiplication_is_communicative}
   (\textbf{Multiplicative communicativity}) The tensor-scalar multiplication $\bullet$ is communicative.
  \end{lemma}
  \begin{proof}
  We show that $\bullet$ is communicative by proving $\alpha \bullet \beta = \beta \bullet \alpha$ for any $\alpha,~\beta \in \mathfrak{R}$. Since the elementwise multiplication $*$ is communicative, i.e., $\mathcal{L}(\alpha) * \mathcal{L}(\beta) = \mathcal{L}(\beta) * \mathcal{L}(\alpha)$, then applying the inverse transform $\mathcal{L}^{-1}$ to both sides, we get $\alpha \bullet \beta = \beta \bullet \alpha$. Therefore, $(\mathfrak{R},~\bullet)$ is an abelian group.
  \end{proof}

  Next we prove that the operator defined in (\ref{def:tensor_scalar_multiplication}) is actually an operation in the space $\mathfrak{R} = \mathbb{R}^{1 \times 1 \times n_3 \times n_4}$, while on the contrary  \cite{kilmer2008third}\cite{kilmer2011factorization} adopted an existing operation (i.e., the circular convolution operation).   Lemma \ref{lemma:tensor_scalar_multiplication_is_an_operation} is the starting point for further definitions including tensor identity, tensor inverse, and tensor eigenvalue.

  \begin{lemma}\label{lemma:tensor_scalar_multiplication_is_an_operation}
  The tensor-scalar multiplication $\bullet$ is an operation in the space $\mathfrak{R}$. Furthermore, $(\mathfrak{R},~\bullet)$ is an abelian group.
  \end{lemma}
  \begin{proof}
  To prove that $\bullet$ is an operation in $\mathfrak{R}$, we need to verify that the tensor-scalar multiplication $\bullet$ satisfies three axioms \cite{pinter2010book}: 1) $\bullet$ is associative, i.e., $(\alpha \bullet \beta) \bullet \gamma = \alpha \bullet (\beta \bullet \gamma)$ for $\alpha,~\beta,~\gamma \in \mathfrak{R}$; 2) the existence of a multiplicative unity, namely, there is a tensor-scalar $e$ in $\mathfrak{R}$ such that $\alpha \bullet e = e \bullet \alpha = \alpha$ for any $\alpha \in \mathfrak{R}$; and 3) the existence of a multiplicative inverse, namely, for every tensor scalar $\alpha \in \mathfrak{R}$, there is an tensor-scalar $\alpha^{-1} \in \mathfrak{R}$  such that $\alpha \bullet \alpha^{-1} = \alpha^{-1} \bullet \alpha = e$. In addition, to prove that $(\mathfrak{R},~\bullet)$ is an abelian group we need to show that the tensor-scalar multiplication $\bullet$ is communicative, i.e., $\alpha \bullet \beta = \beta \bullet \alpha$.

  First, we verify the associativity. Note that the elementwise multiplication $*$ is associative, i.e., $(\mathcal{L}(\alpha) * \mathcal{L}(\beta)) * \mathcal{L}(\gamma) = \mathcal{L}(\alpha) * (\mathcal{L}(\beta) * \mathcal{L}(\gamma))$. Applying the inverse transform $\mathcal{L}^{-1}$ to both sides and combining the definition in (\ref{def:tensor_scalar_multiplication}), we get $(\alpha \bullet \beta) \bullet \gamma = \alpha \bullet (\beta \bullet \gamma)$. Secondly, the existence of a multiplicative unity is verified in Lemma \ref{lemma:multiplicative_unity}.
  Thirdly, we verify the existence of a multiplicative inverse. Let $\alpha^{-1} = \mathcal{L}^{-1}(\mathcal{L}(\alpha)^{-1})$ where $\mathcal{L}(\alpha)^{-1}$ denotes the elementwise inverse of $\mathcal{L}(\alpha)$ that is well-defined. Then, $\alpha \bullet \alpha^{-1} = \mathcal{L}^{-1}(\mathcal{L}(\alpha) * \mathcal{L}(\alpha)^{-1}) = \mathcal{L}^{-1}(\bm{1}) = e$, i.e., $\alpha \bullet \alpha^{-1} = e$. Similarly, we can verify that $\alpha^{-1} \bullet \alpha = e$. Therefore, the space $\mathfrak{R}$ with the tensor-scalar multiplication $\bullet$ is a group.

  Further, Lemma \ref{lemma:multiplication_is_communicative} shows that $\bullet$ is communicative. Therefore, $(\mathfrak{R},~\bullet)$ is an abelian group.
  \end{proof}

  Note that the tensor-scalars play the role of ``scalars" in the space $\mathfrak{R}$. It would be ideal for $\mathfrak{R}$ to be a field, unfortunately, this is not the case as we point out in the following example. Therefore, existing results in the conventional matrix space and the vector space that are defined on fields will not hold. However, we show that it is still able to build a new tensor space to support classic algorithms (SVD, QR, power method and etc.) developed in the conventional matrix space.

  \begin{example}\label{exam:R_is_not_a_field}
  Consider the case $\mathfrak{R} = \mathbb{R}^{1 \times 1 \times 2 \times 2}$ where the tensor-scalars are essentially matrices. Let $\bm{a} = [a~a; a~a]$ such that $a \neq 0$, $\bm{b} = [1~1;-1~-1]$, and the discrete transform $\mathcal{L}$ be the discrete Fourier transform (DFT), then
  \begin{equation}
  \bm{a} \bullet \bm{b} = \bm{0} = \left[
  \begin{array}{cc}
  0 & 0 \\
  0 & 0 \\
  \end{array}
  \right].
  \end{equation}
  For other transforms, one can construct similar examples to show the existence of zero divisors. Therefore, $\mathfrak{R}$ is not a field.
  \end{example}

  \begin{definition}\label{def:tensor_row_column}
  (\textbf{Tensor-column and tensor-row}) We view a fourth-order tensor $\mathcal{A} \in \mathbb{R}^{n_1 \times n_2 \times n_3 \times n_4}$ as an $n_1 \times n_2$ matrix of tensor-scalars, and define the tensor-columns to be $\mathcal{A}(:,j,:,:),~j \in [n_2]$ and the tensor-rows to be $\mathcal{A}(i,:,:,:),~i \in [n_1]$.
  \end{definition}

  The tensor-columns and tensor-rows are essentially ``vectors". For example, the $j$th tensor-column $\mathcal{A}(:,j,:,:)$ is a column vector of tensor-scalars, while the $i$th tensor-row $\mathcal{A}(i,:,:,:)$ is a row vector of tensor-scalars. For easy presentation, we use $\mathcal{A}_j$ to denote $\mathcal{A}(:,j,:,:)$, and $\mathcal{A}_i^T$ to denote $\mathcal{A}(i,:,:,:)$~\footnote{Here, the superscript $^T$ means that we treat $\mathcal{A}_i^T$ as a row vector, as in the matrix case.}. Correspondingly, the space $\mathbb{R}^{n_1 \times n_2 \times n_3 \times n_4}$ is viewed as a matrix space $\mathfrak{R}^{n_1 \times n_2}$.

  \begin{definition}\label{def:square_tensor}
  (\textbf{Square tensor and rectangular tensor}) We view a fourth-order tensor $\mathcal{A} \in \mathbb{R}^{n_1 \times n_2 \times n_3 \times n_4}$ as an $n_1 \times n_2$ matrix in the space $\mathfrak{R}^{n_1 \times n_2}$ with entries being tensor-scalars. If $n_1 = n_2$, we say $\mathcal{A}$ is a square tensor, otherwise we call it a rectangular tensor.
  \end{definition}

  \begin{definition}\label{def:tensor_L_diagonal}
  (\textbf{$\mathcal{L}$-diagonal tensor}) A tensor $\mathcal{A} \in \mathfrak{R}^{n_1 \times n_2}$ is called {\em $\mathcal{L}$-diagonal} if $\mathcal{A}(i,j) = \bm{0}$ for $i \neq j$, where $\bm{0}$ denotes the zero tensor-scalars.
  \end{definition}

  \begin{definition}\label{def:tensor_linear_combination}
  (\textbf{Tensor-linear combinations}) Given tensor scalars $\bm{c}_j \in \mathfrak{R}$, $j \in [n_2]$, a tensor-linear combination of the tensor-columns $\mathcal{A}_j \in \mathfrak{R}^{n_1 \times 1}$, $j \in [n_2]$, is defined as
  \begin{equation}
  \begin{split}
  &\mathcal{A}_1 \bullet \bm{c}_1 + ... + \mathcal{A}_{n_2} \bullet \bm{c}_{n_2} = \mathcal{A} \bullet \bm{c},
  ~\text{where}~\mathcal{A} = [\mathcal{A}_1, ..., \mathcal{A}_{n_2}],
  \bm{c} = \left[
  \begin{array}{c}
  \bm{c}_1\\
  \vdots \\
  \bm{c}_{n_2}
  \end{array}
  \right].
  \end{split}
  \end{equation}
  \end{definition}

  \begin{definition}\label{def:tensor_product}
  (\textbf{Tensor product: $\mathcal{L}$-product}) The $\mathcal{L}$-product $\mathcal{C} = \mathcal{A} \bullet \mathcal{B}$ of $\mathcal{A} \in \mathfrak{R}^{n_1 \times n'}$ and $\mathcal{B} \in \mathfrak{R}^{n' \times n_2}$ is a tensor in $\mathfrak{R}^{n_1 \times n_2}$ (i.e., $\mathbb{R}^{n_1 \times n_2 \times n_3 \times n_4}$), the $(i,j)$th element of $\mathcal{C}$ is defined as follows
  \begin{equation}\label{eq:tenso_scalar_matrix_multiplication}
      \mathcal{C}(i,j) = \sum\limits_{k \in [n']} \mathcal{A}(i, k) \bullet \mathcal{B}(k,j), ~~i \in [n_1],~j\in [n_2].
  \end{equation}
  \end{definition}

  \begin{lemma}\label{lemma:tensor_product_transform_domain_matrix_multiplication}
  The $\mathcal{L}$-product $\mathcal{C} = \mathcal{A} \bullet \mathcal{B}$ can be calculated in the following way:
  \begin{equation}\label{eq:tensor_product_matrix_view}
  \text{bkldiag}(\text{MatView}(\widetilde{\mathcal{C}})) = \text{bkldiag}(\text{MatView}(\widetilde{\mathcal{A}})) \cdot \text{bkldiag}(\text{MatView}(\widetilde{\mathcal{B}})).
  \end{equation}
  Then, we stack the diagonal block matrix $\text{bkldiag}(\text{MatView}(\widetilde{\mathcal{C}}))$ back to tensor $\widetilde{\mathcal{C}}$ and then perform the inverse transform to get $\mathcal{C}$, i.e., $\mathcal{C} = \mathcal{L}^{-1}(\widetilde{\mathcal{C}})$.
  \end{lemma}
  \begin{proof}
  Let us consider the transform-domain representation of (\ref{eq:tenso_scalar_matrix_multiplication}), we have
  \begin{equation}
  \mathcal{L}(\mathcal{C}(i,j))  = \sum\limits_{k \in [n']}  \mathcal{L}(\mathcal{A}(i, k)) * \mathcal{L}(\mathcal{B}(k,j)),~k \in [n'],
  \end{equation}
  which can be represented as $\widetilde{\bm{C}}^p =\widetilde{\bm{A}}^p \widetilde{\bm{B}}^p,~p \in [P]$. Then, according to (\ref{eq:parallel_multiplication_block_structure}), one can easily get (\ref{eq:tensor_product_matrix_view}).
  \end{proof}

  \begin{lemma}\label{lemma:tensor_identity}
  (\textbf{Identity tensor}) The identity tensor $\mathcal{I} \in \mathfrak{R}^{n \times n}$ is an $\mathcal{L}$-diagonal square tensor with $e$'s on the main diagonal and zeros elsewhere, i.e, $\mathcal{I}(i,i) = e$ for $i\in [n]$, where all other entries are zero tensor-scalars $\bm{0}$'s.
  \end{lemma}
  \begin{proof}
  Given any $\mathcal{A} \in \mathfrak{R}^{n \times n}$ and the identity tensor $\mathcal{I} \in \mathfrak{R}^{n \times n}$, we prove that $\mathcal{C} = \mathcal{A} ~\bullet~ \mathcal{I} = \mathcal{I} ~\bullet~ \mathcal{A} = \mathcal{A}$. According to Definition \ref{def:tensor_product}, we have
  \begin{equation}
  \mathcal{C}(i,j) = \sum\limits_{k \in [n]} \mathcal{A}(i, k) \bullet \mathcal{I}(k,j) = \mathcal{A}(i, j) \bullet \mathcal{I}(j,j) = \mathcal{A}(i, j), ~~i \in [n], j \in [n],
  \end{equation}
  which verifies that $\mathcal{C} = \mathcal{A} \bullet \mathcal{I}= \mathcal{A}$. Since Lemma \ref{lemma:tensor_scalar_multiplication_is_an_operation} showed that the tensor-scalar product $\bullet$ is communicative, we can easily verify that $\mathcal{C} = \mathcal{I} \bullet \mathcal{A}= \mathcal{A}$.
  \end{proof}

  \begin{definition}\label{def:tensor_inverse}
  (\textbf{Tensor inverse}) A tensor $\mathcal{A} \in \mathfrak{R}^{n \times n}$ is invertible if there exists a tensor $\mathcal{A}^{-1} \in \mathfrak{R}^{n \times n}$ such that $\mathcal{A} \bullet \mathcal{A}^{-1} = \mathcal{A}^{-1} \bullet \mathcal{A} = \mathcal{I}$. Note that sometimes it is convenient to use the conventional notation $\mathcal{A}^{-1} = \frac{\mathcal{I}}{\mathcal{A}}$.
  \end{definition}

  \begin{definition}\label{def:tensor_column_row_subspaces}
  (\textbf{Tensor-column subspace and tensor-row subspace}) We define the tensor-column subspace of a tensor $\mathcal{A} \in \mathfrak{R}^{n_1 \times n_2}$ to be the space spanned by the tensor-columns $\mathcal{A}_j$, $j \in [n_2]$, and similarly the tensor-row subspace to be the space spanned by the tensor-rows $\mathcal{A}_i^T$, $i \in [n_1]$, denoted by $\text{t-span}(\mathcal{U})$ and $\text{t-span}(\mathcal{V})$, respectively. Formally, $\mathcal{U} \in \mathfrak{R}^{n_2 \times n_2}$ and $\mathcal{V} \in \mathfrak{R}^{n_1 \times n_1}$ are tensors, and the two subspaces can be expressed as follows
  \begin{equation}\label{eq:tensor_column_row_subspaces}
  \begin{split}
  \text{t-span}(\mathcal{U}) &= \left\{\mathcal{X} = \sum\limits_{j \in [n_2]} \mathcal{U}_j \bullet \bm{c}_j,~\bm{c}_j \in \mathfrak{R}\right\} = \left\{\mathcal{Y} = \sum\limits_{j \in [n_2]} \mathcal{A}_j \bullet \bm{d}_j,~\bm{d}_j \in \mathfrak{R}\right\},\\
  \text{t-span}(\mathcal{V}) &= \left\{\mathcal{X} = \sum\limits_{i \in [n_1]} \mathcal{V}_i \bullet \bm{c}_i,~\bm{c}_i \in \mathfrak{R}~\right\} = \left\{\mathcal{Y} = \sum\limits_{i \in [n_1]} \mathcal{A}_i^T \bullet \bm{d}_i,~\bm{d}_i \in \mathfrak{R}\right\}.
  \end{split}
  \end{equation}
  \end{definition}

  Therefore, if we restrict $\mathcal{U}$ and $\mathcal{V}$ to be ``orthogonal" (as formally defined in Definition \ref{def:orthogonal_tensor} in the following), then the tensor-columns $\mathcal{U}_j, ~j\in [n_2]$ and $\mathcal{V}_i,~i \in [n_1]$ are the basis of $\text{t-span}(\mathcal{U})$ and $\text{t-span}(\mathcal{V})$, respectively.

  \begin{lemma}\label{lemma:tensor_hermintian_transpose}
  (\textbf{Tensor Hermintian transpose}) Given $\mathcal{A} \in \mathfrak{R}^{n_1 \times n_2} = \mathbb{R}^{n_1 \times n_2 \times n_3 \times n_4}$, we define the Hermintian transpose $\mathcal{A}^H \in \mathfrak{R}^{n_2 \times n_1}$ such that
  in the two sequences $\text{MatView}(\mathcal{L}(\mathcal{A}^H)) =[\widetilde{\bm{A}^H}^1,...,\widetilde{\bm{A}^H}^P]$ and $\text{MatView}(\mathcal{A}) =[\widetilde{\bm{A}}^1,...,\widetilde{\bm{A}}^P]$ with $P = n_3n_4$, we have $\widetilde{\bm{A}^H}^p = (\widetilde{\bm{A}^p})^H$ for $p \in [P]$.
  Then, the multiplication reversal property of the Hermintian transpose holds, i.e., $\mathcal{A}^H \bullet \mathcal{B}^H = (\mathcal{B} \bullet \mathcal{A})^H$ for $\mathcal{A} \in \mathfrak{R}^{n_1 \times n_2},~\mathcal{B} \in \mathfrak{R}^{n_3 \times n_1}$.
  \end{lemma}
  \begin{proof}
  According to Lemma \ref{lemma:tensor_product_transform_domain_matrix_multiplication}, we have
  \begin{equation}
  \begin{split}
  \mathcal{B} \bullet \mathcal{A} ~\leftrightarrow&~ \text{bkldiag}(\text{MatView}(\widetilde{\mathcal{B}})) \cdot \text{bkldiag}(\text{MatView}(\widetilde{\mathcal{A}})),\\
  (\mathcal{B} \bullet \mathcal{A})^H ~\leftrightarrow&~
  \left( \text{bkldiag}(\text{MatView}(\widetilde{\mathcal{B}})) \cdot \text{bkldiag}(\text{MatView}(\widetilde{\mathcal{A}})) \right)^H \\
  &= \text{bkldiag}(\text{MatView}(\widetilde{\mathcal{A}})^H) \cdot \text{bkldiag}(\text{MatView}(\widetilde{\mathcal{B}})^H),\\
  \end{split}
  \end{equation}
  which gives us $\mathcal{A}^H \bullet \mathcal{B}^H$, according to the definition of the tensor hermintian transpose.
  \end{proof}

  \begin{definition}\label{def:orthogonal_tensor}
  (\textbf{Orthogonal tensor}) A tensor $\mathcal{Q} \in \mathfrak{R}^{n \times n}$ is orthogonal if
  \begin{equation}
  \mathcal{Q}^H \bullet \mathcal{Q} = \mathcal{Q} \bullet \mathcal{Q}^H = \mathcal{I}.
  \end{equation}
  \end{definition}

  \begin{definition}\label{def:symmetric_tensor}
  (\textbf{Symmetric tensor}) A symmetric tensor $\mathcal{A} \in \mathfrak{R}^{n \times n}$ is symmetric if $\mathcal{A}$ can be expressed as $\mathcal{A} = \mathcal{X} \bullet \mathcal{X}^{H}$ where $\mathcal{X} \in \mathfrak{R}^{n \times n}$.
  \end{definition}

  \begin{definition}\label{def:tensor_spectrum_nor}
  (\textbf{Tensor spectrum norm}) The spectrum norm of a tensor $\mathcal{A} \in \mathfrak{R}^{n_1 \times n_2}$ is defined in terms of the $F$-norm of a vector
  \begin{equation}\label{eq:matrix_spectrum_norm}
  ||\mathcal{A}|| = \max\limits_{\bm{x} \in \mathfrak{R}^{n_2 \times 1},~||\bm{x}||_F = 1} ||\mathcal{A} \bullet \bm{x}||_F.
  \end{equation}
  \end{definition}

\subsection{Discussions}
\label{sec:low_tubal_rank_tensor_model}
  We compare our discrete transform-based approach with the low-tubal-rank tensor model \cite{kilmer2008third,kilmer2011factorization,martin2013order}.

\subsubsection{\textbf{Several Operations}}

 For a fourth-order tensor $\mathcal{A} \in \mathbb{R}^{n_1 \times n_2 \times n_3 \times n_4}$, we use the notation $\bm{A}_{(i)} \in \mathbb{R}^{n_1 \times n_2 \times n_3}$ to denote the third-order tensor created by holding the $4$th index of $\mathcal{A}$ fixed at $i$, $i \in [n_4]$. We create the following block circulant representation
  \begin{equation}\label{eq:bcirc}
  \text{bcirc}(\mathcal{A}) = \left[
  \begin{array}{ccccc}
  \bm{A}_{(1)} & \bm{A}_{(n_4)} & \bm{A}_{(n_4 - 1)} & \cdots & \bm{A}_{(2)} \\
  \bm{A}_{(2)} & \bm{A}_{(1)} & \bm{A}_{(n_4)}     & \cdots & \bm{A}_{(3)} \\
  \vdots   & \ddots   & \ddots       & \ddots & \vdots \\
  \bm{A}_{(n_4)} & \bm{A}_{(n_4 - 1)} & \ddots   & \bm{A}_{(2)} & \bm{A}_{(1)}\\
  \end{array}
  \right],
  \end{equation}
  where $\text{bcirc}(\mathcal{A}) \in \mathbb{R}^{n_1n_4 \times n_2n_4 \times n_3}$.
  The $\text{unfold}(\cdot)$ command takes an $n_1 \times n_2 \times n_3 \times n_4$ tensor and returns an $n_1n_4 \times n_2 \times n_3$ block tensor as follows
  \begin{equation}\label{eq:unfold}
  \text{unfold}(\mathcal{A}) = \left[
  \begin{array}{ccccc}
  \bm{A}_{(1)} \\
  \bm{A}_{(2)} \\
  \vdots   \\
  \bm{A}_{(n_4)} \\
  \end{array}
  \right].
  \end{equation}
  The operation that takes $\text{unfold}(\mathcal{A})$ back to tensor $\mathcal{A}$ is the $\text{fold}(\cdot)$ command:
  \begin{equation}\label{eq:fold}
  \text{fold}(\text{unfold}(\mathcal{A})) = \mathcal{A}.
  \end{equation}

\subsubsection{\textbf{The t-product}}

  The key of the low-tubal-rank tensor model  \cite{kilmer2008third,kilmer2011factorization,martin2013order} is a newly introduced multiplication between two tensors, called the \textit{t-product}, that is defined according to an unfolding-folding scheme: 1) unfold the left tensor into a block circulant representation as in (\ref{eq:bcirc}) and the right tensor as a block tensor as in (\ref{eq:unfold}); 2) calculate the \textit{t-product} between those two tensors, by applying the first step again and then we reach the conventional matrix multiplication between two matrices; and 3) fold the result matrix back into a fourth-order tensor.

  \begin{definition}\label{def:t_product}
  (\textbf{t-product} for fourth-order tensors \cite{martin2013order})
  Given two fourth-order tensors $\mathcal{A} \in \mathbb{R}^{n_1 \times n' \times n_3 \times n_4}$ and $\mathcal{B} \in \mathbb{R}^{n' \times n_2 \times n_3 \times n_4}$, the \textit{t-product} $\mathcal{A} *_t \mathcal{B} \in \mathbb{R}^{n_1 \times n_2 \times n_3 \times n_4}$ is defined as follows
  \begin{equation}\label{eq:t_product_unfold_fold}
  \mathcal{A} *_t \mathcal{B} = \text{fold}(\text{bcirc}(\mathcal{A}) *_t \text{unfold}(\mathcal{B})).
  \end{equation}
  Note that (\ref{eq:t_product_unfold_fold}) is recursive because the right-hand side of (\ref{eq:t_product_unfold_fold}) involves a t-product between two third-order tensors, defined as follows
  \begin{equation}
  \begin{split}
  &\mathcal{C} \in \mathbb{R}^{n_1 \times p \times n_3},~\mathcal{D} \in \mathbb{R}^{p \times n_2 \times n_3},\\
  &\mathcal{C} *_t \mathcal{D} = \text{fold}(\text{bcirc}(\mathcal{C}) \cdot \text{unfold}(\mathcal{D})),
  \end{split}
  \end{equation}
  where the right-hand side is the conventional matrix multiplication.
  \end{definition}
  Note that each successive t-product operation involves tensors of one order less, and at the base level we have the conventional matrix multiplication. Therefore, this unfolding-folding computation scheme applies to other higher-order tensors \cite{martin2013order}.

\subsubsection{\textbf{Major drawback (computationally impractical)}}

  However, such a unfolding-folding scheme is computationally impractical for higher-order tensors, due to the following challenges. First, the unfolding operations in (\ref{eq:bcirc})(\ref{eq:unfold}) and the folding operation in (\ref{eq:fold}) are time-consuming. Secondly, the $\text{bcirc}(\cdot)$ operation in (\ref{eq:bcirc}) puts severe challenges in memory since it expands the size exponentially with the order, namely, $\text{bcirc}(\mathcal{A})$ leads to a matrix of size $n_1n_2n_3^2n_4^2$ for $\mathcal{A} \in \mathbb{R}^{n_1 \times n_2 \times n_3 \times n_4}$, and when $n_1 = n_2 = n_3 = n_4 = n$, $\text{bcirc}(\mathcal{A})$ is $n^{2}$ times larger than $\mathcal{A}$. Thirdly, (\ref{eq:t_product_unfold_fold}) requires $O(n_1 n'n_2 n_3^2n_4^2)$ complex-value multiplications.
  For example, given two $100 \times 100 \times 100 \times 100$ tensors $\mathcal{A}$ and $\mathcal{B}$, $\text{bcirc}(\mathcal{A})$ comsumes approximately $8$T Bytes memory where a double-type variable requies $8$ Bytes, while (\ref{eq:t_product_unfold_fold}) requires $10^{14}$ complex-value multiplications.

  For third-order tensors \cite{kilmer2008third,kilmer2011factorization}, the authors exploit the fact that the multiplication of a circulant matrix (determined by its first column vector) and a vector is equivalent to the discrete circular convolution of those two vectors. Thus, the t-product in (\ref{eq:t_product_unfold_fold}) between two third-order tensors is similar to the traditional matrix multiplication between two matrices whose entries are $1 \times 1 \times n$ tensors, where the scalar product is replaced by the discrete circular convolution. More specifically, the t-product in (\ref{eq:t_product_unfold_fold}) can be rewritten as follows
  \begin{definition}
  (\textbf{t-product} for third-order tensors \cite{kilmer2008third,kilmer2011factorization}) The t-product $\mathcal{C} = \mathcal{A} *_t \mathcal{B}$ of $\mathcal{A} \in \mathbb{R}^{n_1 \times n' \times n_3}$ and $\mathcal{B} \in \mathbb{R}^{n' \times n_2 \times n_3}$ is a tensor of size $n_1 \times n_2 \times n_3$,  the $(i,j)$th element of $\mathcal{C}$ is defined as follows
    \begin{equation}\label{eq:t_product_circular_convolution}
      \mathcal{C}(i,j, :) = \sum\limits_{p \in [n']} \mathcal{A}(i, p, :) \odot \mathcal{B}(p,j,:), ~~i \in [n_1],~j\in [n_2],
    \end{equation}
  where $\odot$ is the discrete circular convolution between two same-sized vectors.
  \end{definition}

  \begin{remark}\label{remark:circular_convolution_fourier}
  Let $\bm{F}_n$ be the Discrete Fourier Transform (DFT) matrix, given two vectors $\bm{x}, \bm{y} \in \mathbb{R}^n$, the circular convolution operation can be computed in the following way: $\bm{x} \odot \bm{y} \triangleq \bm{F}_n^{-1} \cdot ((\bm{F}_n \bm{x}) * (\bm{F}_n \bm{y}))$ where $*$ is the Hadmard product (elementwise product). Exploiting the FFT algorithm, the computation complexity of the circular convolution operation can be reduced from $O(n^2)$ to $O(n\log n)$, assuming $n$ is a power of $2$.
  \end{remark}

  \begin{remark}\label{remark:t_product_fourier_transform}
  (\textbf{Transform-based t-product} for third-order tensors) The t-product $\mathcal{C} = \mathcal{A} *_t \mathcal{B}$ of $\mathcal{A} \in \mathbb{R}^{n_1 \times n' \times n_3}$ and $\mathcal{B} \in \mathbb{R}^{n' \times n_2 \times n_3}$ is a tensor of size $n_1 \times n_2 \times n_3$,  the $(i,j)$th element of $\mathcal{C}$ is defined as follows
    \begin{equation}\label{eq:t_product_transform_based_multiplication}
      \mathcal{C}(i,j, :) = \sum\limits_{p \in [n']} \bm{F}_n^{-1} \cdot ((\bm{F}_n \cdot \mathcal{A}(i, p, :)) * (\bm{F}_n \cdot \mathcal{B}(p,j,:))), ~~i \in [n_1],~j\in [n_2].
    \end{equation}
  \end{remark}

  \begin{remark}\label{remark:t_product_transform_is_better}
  The above equations $(\ref{eq:t_product_unfold_fold})$ $(\ref{eq:t_product_circular_convolution})$ and $(\ref{eq:t_product_transform_based_multiplication})$ are three equivalent definitions of the t-product for third-order tensors. Computing $(\ref{eq:t_product_unfold_fold})$ and $(\ref{eq:t_product_circular_convolution})$ require the same amount of complex-value multiplications (i.e., $O(n_1n'n_2n_3^2)$). $(\ref{eq:t_product_unfold_fold})$ requires $O(n_1n'n_3^2 + n'n_2n_3 + n_1n_2n_3)$ memory while $(\ref{eq:t_product_circular_convolution})$ needs $O(n_1n'n_3 + n'n_2n_3 + n_1n_2n_3)$ memory. The transform-based approach (\ref{eq:t_product_transform_based_multiplication}) has advantages in both computation and memory, namely $O(n_1n'n_3 \log n_3 + n'n_2 n_3\log n_3 + n_1n'n_2n_3 + n_1n_2n_3 \log n_3)$ and $O(n_1n'n_3 + n'n_2n_3 + n_1n_2n_3)$, respectively.
  \end{remark}

\section{Fourth-order Tensor SVD Decomposition}
\label{sec:4D_tensor_decomposition}

   First, we show that in the new tensor space defined in Section \ref{sec:new_tensor_space}, any fourth-order tensor can be diagonalized.
   Based on this result, we define the eigenvalue-eigenvector pairs. The tensor-eigenvalue equation and tensor-eigenvector equation are no longer equivalent, which is essentially different from the conventional matrix case. This fact leads to stronger conditions to guarantee uniqueness of tensor SVD. Then, we define a new tensor SVD with a corresponding algorithm, and also a much simpler condition as a positive indicator when the algorithm output is a unique decomposition.

\subsection{Tensor Diagonalization and Eigenvalue-Eigenvector}

  Recall that eigenvalues of matrices are the roots of the characteristic polynomial $\textbf{det}(\bm{A} - \lambda \bm{I}) = 0$. The existence of eigenvalues implies the existence of corresponding eigenvectors $\bm{x} \in \mathbb{R}^{n \times 1}$.

  \begin{definition}
  (\textbf{Tensor determinant}) The determinant of a square tensor is computed similar to the determinant of a square matrix, except that we replace the real-value scalar multiplication by the tensor-scalar multiplication, e.g., $\mathcal{A} \in \mathfrak{R}^{2 \times 2}$,
  \begin{equation}
  \textbf{det}(\mathcal{A})=  \left|
  \begin{array}{cc}
  \bm{a} & \bm{b} \\
  \bm{c} & \bm{d} \\
  \end{array}
  \right| = \bm{a} \bullet \bm{d} - \bm{b} \bullet \bm{c}.
  \end{equation}
  \end{definition}

  Consider a square tensor $\mathcal{A} \in \mathfrak{R}^{n \times n}$ and $ \bm{\lambda} \in \mathfrak{R}$,  then the characteristic polynomial in our new tensor space becomes as follows
  \begin{equation}\label{eq:determinant_equation}
  \textbf{det}(\mathcal{A} - \bm{\lambda} \bullet \mathcal{I}) = \bm{0}.
  \end{equation}
  In the following, we then formally describe the tensor-eigenvalue equation where the tensor-eigenvalue and the tensor-eigenvector are defined.

  \begin{definition}\label{def:eigenvalue_eigenvector}
  Given a square tensor $\mathcal{A} \in \mathfrak{R}^{n \times n}$, we define the tensor-eigenvalue as $\bm{\lambda} \in \mathfrak{R}$ and the tensor-eigenvector as $\bm{x} \in \mathfrak{R}^{n \times 1}$ such that the following tensor-eigenvalue equation holds
  \begin{equation}\label{eq:eigenvalue_eigenvector}
  \mathcal{A} \bullet \bm{x} = \bm{\lambda} \bullet \bm{x}.
  \end{equation}
  \end{definition}

  \begin{remark}\label{remark:tensor_has_more_eigenvectors}
  In contrast to the matrix case, the determinant equation (\ref{eq:determinant_equation}) and the eigenvector equation (\ref{eq:eigenvalue_eigenvector}) are no longer equivalent in the space $(\mathfrak{R},~\bullet)$. E.g., in the transform domain, let us set $\widetilde{\bm{x}}_1$, $\widetilde{\lambda}_1$ to be an eigenpair of $\widetilde{\bm{A}}_1$, and $\widetilde{x}_p = 0$ for $p = 2:n_3n_4$, then any value of $\widetilde{\lambda}_p$ fits into (\ref{eq:eigenvalue_eigenvector}). However, only a subset of these solutions will also satisfy (\ref{eq:determinant_equation}). Therefore, a big difference from the conventional matrix space is that: a square tensor $\mathcal{A} \in \mathfrak{R}^{n \times n}$ may have more than $n$ pairs of tensor-eigenvalue and tensor-eigenvector satisfying (\ref{eq:eigenvalue_eigenvector}).
  \end{remark}

  Lemma \ref{lemma:tensor_diagonalization} shows the existence of a diagonalization equation. Although this lemma does not provide a method to diagonalize a tensor, it shows that (\ref{eq:eigenvalue_eigenvector}) can be extracted from the diagonalizing equation (\ref{eq:diagonalization_equation}).
  \begin{lemma}\label{lemma:tensor_diagonalization}
  (\textbf{Tensor diagonalization}) Given a square tensor $\mathcal{A} \in \mathfrak{R}^{n \times n}$, there exists a tensor $\mathcal{X} \in \mathfrak{R}^{n \times n}$ and an $\mathcal{L}$-diagonal tensor $\mathcal{D} \in \mathfrak{R}^{n \times n}$, such that
  \begin{equation}\label{eq:diagonalization_equation}
  \mathcal{A} \bullet \mathcal{X} = \mathcal{X} \bullet \mathcal{D}.
  \end{equation}
  \end{lemma}
  \begin{proof}
  Considering the $i$th tensor-column, we have the following equivalent forms
  \begin{equation}\label{eq:tensor_diagonalization_formula_1}
  \begin{split}
  [\mathcal{A} \bullet \mathcal{X}]_i &= \mathcal{A} \bullet \mathcal{X}_i, ~i\in[n].
  \end{split}
  \end{equation}
  In the transform-domain, we know that there exists an eigenvalue-eigenvector pair for each matrix $\widetilde{\bm{A}}^p,~p\in[n_3n_4]$, i.e., $\widetilde{\bm{A}}^p \widetilde{\bm{x}}_i^p = \lambda^p \widetilde{\bm{x}}_i^p$. Transforming back to the time-domain, we know that there exists at least one eigenvalue-eigenvector pair (as pointed in Remark \ref{remark:tensor_has_more_eigenvectors}), $\mathcal{X}_i = \mathcal{L}^{-1}( \text{fold}([\widetilde{\bm{x}}_i^1,...,\widetilde{\bm{x}}_i^{n_3n_4}]))$ and $\bm{\lambda}_{i} = \mathcal{L}^{-1}([\lambda^1,...,\lambda^{n_3n_4}])$,
  such that for $i \in [n]$,
  \begin{equation}
  \mathcal{A} \bullet \mathcal{X}_i = \bm{\lambda}_{i} \bullet \mathcal{X}_i.
  \end{equation}
  Putting $\bm{\lambda}_{i}$ in the diagonal of an $n \times n$ tensor $\mathcal{D} \in \mathfrak{R}^{n \times n}$, and $\mathcal{X}_i$ as tensor-columns of an $n \times n$ tensor $\mathcal{X}$, then we have
  \begin{equation}
  \begin{split}
  \mathcal{A} \bullet \mathcal{X}_i = \mathcal{D}_{ii} \bullet \mathcal{X}_i = \mathcal{X} \bullet \mathcal{D}_i,
  \end{split}
  \end{equation}
  Therefore, we obtain
  \begin{equation}
  \mathcal{A} \bullet \mathcal{X} = \mathcal{X} \bullet \mathcal{D}.
  \end{equation}
  \end{proof}

  \begin{lemma}\label{lemma:matrix_eigenvector_exists}
  \cite{golub2012matrix}
  Given a matrix $\bm{A} \in \mathbb{R}^{n_1 \times n_2}$, there exists a unit vector $\bm{z} \in \mathbb{R}^{n_2 \times 1}$ such that $\bm{A}^T \bm{A} \bm{z} = \mu^2 \bm{z}$ where $\mu = ||\bm{A}||$.
  \end{lemma}

  \begin{lemma}\label{lemma:tensor_eigenvector_exists}
  Given a tensor $\mathcal{A} \in \mathfrak{R}^{n_1 \times n_2}$, then there exists a unit $F$-norm tensor $\bm{z} \in \mathfrak{R}^{n_2 \times 1}$ such that $\mathcal{A}^{H} \bullet \mathcal{A} \bullet \bm{z} = \mu^2 \bm{z}$ where $\mu = ||\mathcal{A}||$ where the tensor spectrum is defined in Definition \ref{def:tensor_spectrum_nor}.
  \end{lemma}
  \begin{proof}
  We can represent $\mathcal{A}^{H} \bullet \mathcal{A} \bullet \bm{z}$ in the block diagonal matrix form (\ref{eq:blkdiagonal_presentation_multiplication_two_matrix_sequence}) as follows
  \begin{equation}
  \begin{split}
    \mathcal{A}^{H} \bullet \mathcal{A} \bullet \bm{z} ~&\leftrightarrow~ \text{blkdiag}(\text{MatView}(\widetilde{\mathcal{A}^H})) \cdot \text{blkdiag}(\text{MatView}(\widetilde{\mathcal{A}})) \cdot \text{blkdiag}(\text{MatView}(\widetilde{\bm{z}})),\\
    &\leftrightarrow~(\widetilde{\bm{A}^H})^p \widetilde{\bm{A}}^p \widetilde{\bm{z}}^p = (\widetilde{\bm{A}^p})^H \widetilde{\bm{A}}^p \widetilde{\bm{z}}^p,~p \in [P],
  \end{split}
  \end{equation}
  where the last equation follows from Lemma \ref{lemma:tensor_hermintian_transpose}.
  Note that Lemma \ref{lemma:matrix_eigenvector_exists} proved the existence of a unite vector $\widetilde{\bm{z}}^p$ such that
  \begin{equation}
  \begin{split}
  (\widetilde{\bm{A}^p})^H \widetilde{\bm{A}}^p \widetilde{\bm{z}}^p = \mu^2 \widetilde{\bm{z}}^p,~~\text{with}~\mu = ||\widetilde{\bm{A}}^p||,
  \end{split}
  \end{equation}
  where we verify that $\mu = ||\mathcal{A}||$ according to the tensor spectrum norm in Definition \ref{def:tensor_spectrum_nor}.
  \end{proof}

\subsection{Tensor $\mathcal{L}$-SVD}

  We first consider the eigendecomposition for symmetric tensors, because symmetry guarantees that there is an orthogonal basis for eigenvectors. Then based on Theorem \ref{theorem:symmetric_tensor_eigendecompsotion}, we prove Theorem \ref{theorem:tensor_L_SVD}.

  \begin{theorem}\label{theorem:symmetric_tensor_eigendecompsotion}
  Let $\mathcal{A} \in \mathfrak{R}^{n \times n}$ be a symmetric tensor. Then, there exists an orthogonal $\mathcal{Q}$ such that $\mathcal{A}$ can be expresses as the following diagonal form
  \begin{equation}
  \mathcal{Q}^H \bullet \mathcal{A} \bullet \mathcal{Q} = \mathcal{S} = \text{diag}(\bm{\lambda}_1, ..., \bm{\lambda}_n).
  \end{equation}
  where $\mathcal{A} \bullet \mathcal{Q}_i = \bm{\lambda}_i \bullet \mathcal{Q}_i$ for $i \in [n]$, and the tensor-columns $\mathcal{Q}_i$ can be permutated so that the tensor-eigenvlues $\bm{\lambda}_i$ appear correspondingly along the diagonal of $\mathcal{S}$.
  \end{theorem}
  \begin{proof}
  Lemma \ref{lemma:tensor_eigenvector_exists} says that, given $\bm{\lambda}_1$ and tensor $\mathcal{A} \in \mathfrak{R}^{n \times n}$, there exists a unit $F$-norm tensor-eigenvector such that $\mathcal{A} \bullet \bm{x} = \bm{\lambda}_1 \bullet \bm{x}$. There exists a tensor $\mathcal{P}_1 \in \mathfrak{R}^{n \times n}$ such that $\mathcal{P}_1^H \bullet \bm{x} = \mathcal{I}_1$ (since $\bm{x}$ behaves as a ``vector" in the space $\mathfrak{R}^{n \times n}$ and is essentially rank-$1$, and $\mathcal{I}_1$ is the first tensor-column of the identity tensor $\mathcal{I}$). Therefore, from $\mathcal{A} \bullet \bm{x} = \bm{\lambda}_1 \bullet \bm{x}$ we get that $(\mathcal{P}_1^H \bullet \mathcal{A} \bullet \mathcal{P}_1) \bullet \mathcal{I}_1 = \bm{\lambda}_1 \bullet \mathcal{I}_1$, namely the first tensor-column of $\mathcal{P}_1^H \bullet \mathcal{A} \bullet \mathcal{P}_1$ is a tensor-scalar multiplication of $\bm{\lambda}_1$ and $\mathcal{I}_1$.
  Since $\mathcal{A}$ is symmetric, we then have the form
  \begin{equation}
  \mathcal{P}_1^H \bullet \mathcal{A} \bullet \mathcal{P}_1 = \left[
  \begin{array}{cc}
  \bm{\lambda}_1 & \bm{0} \\
  \bm{0} & \mathcal{A}(2:n, 2:n)
  \end{array}
  \right],
  \end{equation}
  where $\mathcal{A}(2:n, 2:n) \in \mathfrak{R}^{(n-1) \times (n-1)}$ is also symmetric.

  Next, we apply the induction method by assuming that there is an orthogonal tensor $\mathcal{Q}(2:n,2:n) \in \mathfrak{R}^{(n-1) \times (n-1)}$ such that $\mathcal{Q}(2:n,2:n)^H \bullet \mathcal{A}(2:n, 2:n) \bullet \mathcal{Q}(2:n,2:n) = \text{diag}(\bm{\lambda}_2,...,\bm{\lambda}_n)$. Let us set
  \begin{equation}
  \mathcal{Q} = \mathcal{P}_1 \left[
  \begin{array}{cc}
  e & \bm{0} \\
  \bm{0} & \mathcal{Q}(2:n,2:n)
  \end{array}
  \right], ~~~\mathcal{S} = \left[
  \begin{array}{cc}
  \bm{\lambda}_1 & \bm{0} \\
  \bm{0} & \text{diag}(\bm{\lambda}_2,...,\bm{\lambda}_n)
  \end{array}
  \right],
  \end{equation}
  and then comparing the tensor-columns of the equation $\mathcal{A} \bullet \mathcal{Q} = \mathcal{Q} \bullet \mathcal{S}$, then the theorem is proved.
  \end{proof}

  \begin{theorem}\label{theorem:tensor_L_SVD}
  (\textbf{$\mathcal{L}$-SVD}) For a tensor $\mathcal{A} \in \mathfrak{R}^{n_1 \times n_2}$, the $\mathcal{L}$-SVD of $\mathcal{A}$ is given by
  \begin{equation}\label{eq:tensor_L_SVD}
  \mathcal{A} = \mathcal{U} \bullet \mathcal{S} \bullet \mathcal{V}^H,
  \end{equation}
  where the orthogonal tensors $\mathcal{U} \in \mathfrak{R}^{n_1 \times n_1}$ and $\mathcal{V} \in \mathfrak{R}^{n_2 \times n_2}$ are the tensor-column subspace and tensor-row subspace, respectively, and $\mathcal{S} \in \mathfrak{R}^{n_1 \times n_2}$ is a rectangular $\mathcal{L}$-diagonal tensor.
  \end{theorem}
  \begin{proof}
  We say that (\ref{eq:tensor_L_SVD}) gives the right SVD decomposition of a tensor if the following aspects hold (similar to the matrix SVD \cite{golub2012matrix}): 1) the tensor-columns of orthogonal tensors $\mathcal{U} \in \mathfrak{R}^{n_1 \times n_1}$ and $\mathcal{V} \in \mathfrak{R}^{n_2 \times n_2}$ correspond to the tensor-eigenvectors of $\mathcal{A}^H \bullet \mathcal{A}$ and $\mathcal{A} \bullet \mathcal{A}^H$, respectively, and 2) the tensor-eigenvalues on the diagonal of $\mathcal{S} \in \mathfrak{R}^{n_1 \times n_2}$ are the square roots of the nonzero tensor-eigenvalues of both  $\mathcal{A}^H \bullet \mathcal{A}$ and $\mathcal{A} \bullet \mathcal{A}^H$.

  We consider a tensor constructed as $\mathcal{A}^H \bullet \mathcal{A}$, which is a symmetric tensor according to Definition \ref{def:symmetric_tensor}. Theorem \ref{theorem:symmetric_tensor_eigendecompsotion} states that $\mathcal{A}^H \bullet \mathcal{A}$ has a complete set of tensor-eigenvectors that correspond to a basis of the tensor-column subspace, i.e., $\mathcal{A}^H \bullet \mathcal{A} \bullet \bm{x}_j  = \bm{\lambda}_j \bullet \bm{x}_j$. Therefore, we have
  \begin{equation}
  \bm{x}_i^H \bullet (\mathcal{A}^H \bullet \mathcal{A} \bullet \bm{x}_j) = \bm{\lambda}_j \bullet (\bm{x}_i^H \bullet \bm{x}_j) = \bm{\lambda}_j \bullet \bm{\delta}_{ij},
  \end{equation}
  where $\bm{\delta}_{ij}$ is the extended Kronecker delta such that if $i=j$, $\bm{\delta}_{ij} = e$.

  Define $\bm{\sigma}_j = \sqrt{\bm{\lambda}_j}$ and $\bm{q}_j = \mathcal{A} \bullet \bm{x}_j \bullet \bm{\sigma}_j^{-1}$, then we have
  \begin{equation}
  \bm{q}_i^H \bullet \bm{q}_j = (\bm{\sigma}_i^{-1} \bullet \bm{x}_i^H \bullet \mathcal{A}^H) \bullet \mathcal{A} \bullet \bm{x}_j \bullet \bm{\sigma}_j^{-1} = \bm{\delta}_{ij}.
  \end{equation}
  We organize $\bm{x}_j$'s into $\mathcal{V}$ and $\bm{q}_j$'s into $\mathcal{U}$, then
  \begin{equation}
  (\mathcal{U}^H \bullet \mathcal{A} \bullet \mathcal{V})_{ij} = \bm{q}_i^H \bullet \mathcal{A} \bullet \bm{x}_j = \left\{
  \begin{array}{cc}
  \bm{0},~~~~~~~~~~~~~&\text{if}~j > r,\\
  \bm{\sigma_j}\bullet \bm{q}_i^H \bullet \bm{q}_j = \bm{\sigma}_j \bullet \bm{\delta}_{ij}, ~&\text{if}~j \leq r,
  \end{array}
  \right.
  \end{equation}
  which is $\mathcal{U}^H \bullet \mathcal{A} \bullet \mathcal{V} = \mathcal{S}$. Note that the tensor-eigenvalues on the diagonal $\bm{\sigma}_j$ are the square roots of the nonzero tensor-eigenvalues $\bm{\lambda}_j$ of $\mathcal{A}^H \bullet \mathcal{A}$.

  For the case of $\mathcal{A} \bullet \mathcal{A}^H$, the verification procedure is the same. Therefore, we get  $\mathcal{A} = \mathcal{U} \bullet \mathcal{S} \bullet \mathcal{V}^H.$
  \end{proof}

  \begin{definition}
  (\textbf{Tensor $\mathcal{L}$-rank}) The tensor $\mathcal{L}$-rank of $\mathcal{A} \in \mathfrak{R}^{n_1 \times n_2}$, denoted as $\mathcal{L}$-rank, is defined to be the number of non-zero tensor-scalars of $\mathcal{S}$ in the $\mathcal{L}$-SVD factorization (\ref{eq:tensor_L_SVD}).
  \end{definition}

\subsection{Algorithm for Computing $\mathcal{L}$-SVD}

 \begin{algorithm}[t]
  \caption{$\text{$\mathcal{L}$-SVD}(\mathcal{A})$}
  \begin{algorithmic}
   \label{alg:L_SVD}
  \STATE \textbf{Input: } $\mathcal{A} \in \mathfrak{R}^{n_1 \times n_2}$
  \STATE $\widetilde{\mathcal{A}} = \mathcal{L}(\mathcal{A})$,
  \STATE For the matrix sequence $\text{MatView}(\widetilde{\mathcal{A}}) = \{\bm{\widetilde{A}}^1,...,\bm{\widetilde{A}}^P\},~~P=n_3n_4$;
  \FOR{$p = 1~\text{to}~P$}
  	\STATE $[\widetilde{\mathcal{U}}^{p}, \widetilde{\mathcal{S}}^{p}, \widetilde{\mathcal{V}}^{p}] = {\tt SVD} (\widetilde{\bm{A}}^{p})$;
  \ENDFOR
  \STATE $\mathcal{U} \leftarrow \mathcal{L}^{-1}(\widetilde{\mathcal{U}})$, $\mathcal{S} \leftarrow \mathcal{L}^{-1}(\widetilde{\mathcal{S}})$, $\mathcal{V}^{H} \leftarrow \mathcal{L}^{-1}(\widetilde{\mathcal{V}})$.
  \end{algorithmic}
\end{algorithm}

  Now we provide an algorithm to compute the $\mathcal{L}$-SVD, as shown in Alg. \ref{alg:L_SVD}. The basic flow of Alg. \ref{alg:L_SVD} consists of three steps: 1) perform a multidimensional discrete transform $\mathcal{\mathcal{A}}$ to get $\widetilde{\mathcal{A}}$ and forming the matrix view $\text{MatView}(\widetilde{\mathcal{A}})$; 2) execute the conventional matrix SVD on each element of the matrix sequence $\text{MatView}(\widetilde{\mathcal{A}})$; and 3) form the factors $\widetilde{\mathcal{U}},~\widetilde{\mathcal{S}}$ and $\widetilde{\mathcal{V}}$, and then transform them back to the time-domain.

  \begin{theorem}\label{theorem:L_SVD}
   Given $\mathcal{A} \in \mathbb{R}^{n_1 \times n_2 \times n_3}$, then Alg. \ref{alg:L_SVD} outputs an $\mathcal{L}$-SVD as follows
   \begin{equation}\label{eq:tensor_L_SVD_representation}
   \mathcal{A} = \mathcal{U} ~\bullet~ \mathcal{S} ~\bullet~ \mathcal{V}^{H},
   \end{equation}
   where $\mathcal{U} \in \mathfrak{R}^{n_1 \times n_1}$ and $\mathcal{V} \in \mathbb{R}^{n_2 \times n_2}$ are orthogonal respectively, and $\mathcal{S} \in \mathbb{R}^{n_1 \times n_2}$ is $\mathcal{L}$-diagonal.
   \end{theorem}

   \begin{proof}
   We show that Alg. \ref{alg:L_SVD} outputs one $\mathcal{L}$-SVD decomposition by constructing a correspondence between $\mathcal{L}$-SVD and matrix SVD. We verify that $\mathcal{U} \in \mathfrak{R}^{n_1 \times n_1}$ and $\mathcal{V} \in \mathbb{R}^{n_2 \times n_2}$ are orthogonal, and $\mathcal{S} \in \mathbb{R}^{n_1 \times n_2}$ is $\mathcal{L}$-diagonal.

   As shown in Alg. \ref{alg:L_SVD}, the matrix SVD $\widetilde{\bm{A}}^{p} = \widetilde{\bm{U}}^{p} \widetilde{\bm{S}}^{p} \widetilde{\bm{V}}^{p}$ for $p \in [P]$ with $P=n_3n_4$. We represent those $n_3n_4$ parallel SVDs using the following block diagonal matrices
   \begin{equation}\label{eq:tensor_L_SVD_block_matrix_view}
   \left[
   \begin{array}{ccc}
    \widetilde{\bm{A}}^{1} &  & \\
    & \ddots & \\
    & & \widetilde{\bm{A}}^{P}\\
    \end{array}
    \right] =  \left[
   \begin{array}{ccc}
    \widetilde{\bm{U}}^{1} &  & \\
    & \ddots & \\
    & & \widetilde{\bm{U}}^{P}\\
    \end{array}
    \right]  \left[
   \begin{array}{ccc}
    \widetilde{\bm{S}}^{1} &  & \\
    & \ddots & \\
    & & \widetilde{\mathcal{S}}^{P}\\
    \end{array}
    \right] \left[
   \begin{array}{ccc}
    \widetilde{\mathcal{V}}^{1} &  & \\
    & \ddots & \\
    & & \widetilde{\mathcal{V}}^{P}\\
    \end{array}
    \right],
   \end{equation}
  which can be properly represented as
  \begin{equation}
  \text{blkdiag}(\text{MatView}(\mathcal{A})) = \text{blkdiag}(\text{MatView}(\mathcal{U})) \cdot \text{blkdiag}(\text{MatView}(\mathcal{S})) \cdot \text{blkdiag}(\text{MatView}(\mathcal{V})),
  \end{equation}
   We perform the inverse transform $\mathcal{L}^{-1}$ on $\widetilde{\mathcal{U}}$, and $\widetilde{\mathcal{S}}$, we get $\mathcal{U}$ and $\mathcal{S}$, respectively. Note that for $\widetilde{\mathcal{V}}$, it involves the tensor Hermintian transpose as given in Lemma \ref{lemma:tensor_hermintian_transpose}.

   We then show that $\mathcal{U}$ and $\mathcal{V}$ are orthogonal, relying on the forward and backward transform $\mathcal{L}$ and $\mathcal{L}^{-1}$. We first show that $\mathcal{U}^H ~\bullet~ \mathcal{U} =  \mathcal{I}$. From the matrix SVD, we already know that each $\widetilde{U}^{p}$ is orthgonal, for $p \in [P]$ with $P=n_3n_4$. In the block diagonal matrix form, we have
   \begin{equation}\label{eq:tensor_L_SVD_orthogonal}
   \left[
   \begin{array}{ccccc}
    \widetilde{\mathcal{U}^{H}}^{1} &  & & &\\
    & &  \ddots & & \\
    & & & & \widetilde{\mathcal{U}^{H}}^{P}\\
    \end{array}
    \right]
   \left[
   \begin{array}{ccccc}
    \widetilde{\mathcal{U}}^{1} & &  & &\\
    & & \ddots & & \\
    & & & & \widetilde{\mathcal{U}}^{P}\\
    \end{array}
    \right] = \left[
   \begin{array}{ccccc}
    \bm{I} &  & & & \\
    & & \ddots & &\\
    & & & & \bm{I} \\
    \end{array}
    \right],
   \end{equation}
   where the equality follows from the definition of the tensor Hermintian transpose in Lemma \ref{lemma:tensor_hermintian_transpose}. Note that (\ref{eq:tensor_L_SVD_orthogonal}) can be transformed back to the target equality $\mathcal{U}^H ~\bullet~ \mathcal{U} =  \mathcal{I}$.

   Following the above forward-backward transform scheme, one can easily verify that $\mathcal{U} ~\bullet~ \mathcal{U}^H =  \mathcal{I}$,  $\mathcal{V}^H ~\bullet~ \mathcal{V} =  \mathcal{I}$, and $\mathcal{V} ~\bullet~ \mathcal{V}^H =  \mathcal{I}$, meaning that both $\mathcal{U}$ and $\mathcal{V}$ are $\mathcal{L}$-based orthgonal tensors. Moreover, each matrix $\widetilde{\bm{S}}^{p} $ in (\ref{eq:tensor_L_SVD_block_matrix_view}) is diagonal, and then we know that $\mathcal{S}$ is $\mathcal{L}$-diagonal according to Definition \ref{def:tensor_L_diagonal}.
   \end{proof}

  Remark \ref{remark:tensor_has_more_eigenvectors} state that there are so many tensor-eigenvalues and tensor-eigenvectors. Here, we provide a much simpler condition as a positive indicator when the output of Alg. \ref{alg:L_SVD} is actually a unique decomposition. Note that if the conditions required by Theorem \ref{theorem:unique_canonical_eigendecomposition} do not hold, then we cannot the output $\mathcal{L}$-SVD is not unique.

  \begin{definition}\label{def:tensor_eigenvalue_eigenvector_canonical_set}
  \textbf{(Canonical set of tensor-eigenvalues and tensor-eigenvectors)} Given a tensor $\mathcal{A} \in \mathfrak{R}^{n_1 \times n_2}$, a canonical set of tensor-eigenvalues and tensor-eigenvectors is a set of minimum size, ordered such that $\text{abs}(\lambda_1) \succeq \text{abs}(\lambda_2) \succeq ... \succeq \text{abs}(\lambda_n)$, which contains the information to reproduce any other pair of tensor-eigenvalue and tensor-eigenvector of $\mathcal{A}$ that satisfies (\ref{eq:eigenvalue_eigenvector}).
  \end{definition}

  Next, Theorem \ref{theorem:unique_canonical_eigendecomposition} states that if there exists a canonical set of tensor-eigenvalues and tensor-eigenvectors, then the $\mathcal{L}$-SVD returned by Alg. \ref{alg:L_SVD} is unique.

  \begin{lemma}\label{lemma:tensor_basis_condition}
  Given $\mathcal{X} \in \mathfrak{R}^{n_1 \times n_2}$, then $\mathcal{X} $ spans $\mathfrak{R}^{n_1}$ if and only if $\text{bkldiag}(\text{MatView}(\widetilde{X}))$ has rank $n_1P$. Moreover, $\mathcal{X} $ is a basis if and only if $\text{bkldiag}(\text{MatView}(\widetilde{\mathcal{X}}))$ are invertible.
  \end{lemma}
  \begin{proof}
  Assume that $\bm{y} = \mathcal{X} \bullet \bm{a}$ where $\bm{y} \in \mathfrak{R}^{n_1},~\bm{a} \in \mathfrak{R}^{n_2}$, we have
  \begin{equation}\label{eq:existence_tensor_basis}
  \text{bkldiag}(\text{MatView}(\widetilde{\bm{y}})) = \text{bkldiag}(\text{MatView}(\widetilde{\mathcal{X}})) \cdot \text{bkldiag}(\text{MatView}(\widetilde{\bm{a}})).
  \end{equation}
  Assume that (\ref{eq:existence_tensor_basis}) holds for any $\bm{y} \in \mathfrak{R}^{n_1}$, then each matrix $\bm{\widetilde{X}}^p$ must have rank $n_1$, namely, $\text{bkldiag}(\text{MatView}(\widetilde{X}))$ has rank $n_1P$.

  Conversely, if $\text{bkldiag}(\text{MatView}(\widetilde{\mathcal{X}}))$ has rank $n_1P$ (being invertible), then each matrix $\bm{\widetilde{X}}^p$ must have rank $n_1$, and thus (\ref{eq:existence_tensor_basis}) holds for any $\bm{y} \in \mathfrak{R}^{n_1}$. Therefore, $\mathcal{X} $ is a basis for the space $\mathfrak{R}^{n_1}$.
  \end{proof}

  \begin{theorem}\label{theorem:unique_canonical_eigendecomposition}
  (\textbf{Unique canonical $\mathcal{L}$-SVD}) Given a tensor $\mathcal{A} \in \mathfrak{R}^{n_1 \times n_2}$, $\mathcal{A}$ has a unique canonical set of $n = \min(n_1,n_2)$ tensor-eigenvalues and tensor-eigenvalues if the matrices $\widetilde{\bm{A}}^{p},~p\in[n_3n_4]$ have distinct eigenvlues with distince magnitudes. This canonical set of tensor-eigenvectors corresponds to a basis of the tensor-column subspace and also the tensor-row subspace, yielding an tensor-eigendecomposition as (\ref{eq:tensor_L_SVD_representation}).
  \end{theorem}
  \begin{proof}
  Since all the tensor-eigenvalues of $\widetilde{\mathcal{A}}_p,~p\in[n_3n_4]$ are distinct with distinct magnitudes, there are $nn_3n_4$ distinct numbers. It implies that any unique canonical set must have at least $n$ tensor-eigenvalues and correspondingly $n$ tensor-eigenvectors. Let $\widetilde{\lambda}_i^p$ ($i\in[n],~p\in[P]$ with $P=n_3n_4$) be the $i$th tensor-eigenvalue of $\widetilde{\bm{A}}^p$, with the ordering $|\widetilde{\lambda}_i^1| > ...> \widetilde{\lambda}_i^p > ... > |\widetilde{\lambda}_i^{P}|$. Then $\bm{\lambda}_i = \mathcal{L}^{-1}(\widetilde{\lambda}_i^1,...,\widetilde{\lambda}_i^P)$ is a canonical set of tensor-eigenvalues.

  Now we show that this set constitutes an tensor-eigenbasis. Consider the matrix SVD $\widetilde{\bm{A}}^{p} = \widetilde{\bm{U}}^{p} \widetilde{\bm{S}}^{p} \widetilde{\bm{V}}^{p}$ for $p \in [P]$ with the ordering of eigenvalues. Then, we construct $\mathcal{U}$ and $\mathcal{V}^H$ as in Alg. \ref{alg:L_SVD}. Since all the tensor-eigenvalues of $\widetilde{\mathcal{A}}_p,~p\in[n_3n_4]$ are distinct with distinct magnitudes, then $\text{bkldiag}(\text{MatView}(\widetilde{\mathcal{U}}))$ and $\text{bkldiag}(\text{MatView}(\widetilde{\mathcal{V}}))$ has rank at least $nP$ with $n = \min(n_1,n_2)$. According to Lemma \ref{lemma:tensor_basis_condition}, we know that $\mathcal{U}$ and $\mathcal{V}^H$ are a basis of the tensor-column subspace and the tensor-row subspace.

  Finally, we show that this set is unique. For any canonical set, according to Definition \ref{def:tensor_eigenvalue_eigenvector_canonical_set}, we know that $\text{abs}(\bm{\lambda}_1) \succeq \text{abs}(\bm{\lambda}_i)$ for $i=2:n$. Recalling the ordering (\ref{eq:ordering_tensor_scalar}), we get $\bm{\widetilde{\lambda}}_1^p \geq \bm{\widetilde{\lambda}}_i^p$ for $i=2:n$. Since all the values $\bm{\widetilde{\lambda}}_i^p$ are unique, we have $\bm{\widetilde{\lambda}}_1^p > \bm{\widetilde{\lambda}}_i^p$ for $i=2:n$. Therefore, there is no other choice of $\bm{\widetilde{\lambda}}_1^p$ in a canonical set, namely, $\bm{\lambda}_1$ is unique. Repeating the above argument on the rest tensor-eigenvalues $\bm{\lambda}_i$ for $i=2:n$, then the uniqueness is verified.
  \end{proof}

\section{$\mathcal{L}$-QR Decomposition}

  We define the $\mathcal{L}$-QR for fourth-order tensors. We propose a Householder transformation-based algorithm that outputs the expected triangular factorization.

\subsection{Tensor QR: $\mathcal{L}$-QR}

  \begin{definition}\label{def:matrix_QR}
  (\textbf{Matrix QR})
  The QR factorization of an $n_1 \times n_2$ matrix $\bm{A}$ is given by
  \begin{equation}
  \bm{A} = \bm{Q} \bm{R},
  \end{equation}
  where $\bm{Q} \in \mathbb{R}^{n_1 \times n_1}$ is orthogonal, $\bm{R} \in \mathbb{R}^{n_1 \times n_2}$ is upper triangular, and we assume $n_1 \geq n_2$.
  \end{definition}

  \begin{definition}\label{def:tensor_L_QR}
  (\textbf{$\mathcal{L}$-QR})
  The $\mathcal{L}$-QR factorization of a tensor $\mathcal{A} \in \mathfrak{R}^{n_1 \times n_2}$ is given by
  \begin{equation}
  \mathcal{A} = \mathcal{Q} \bullet \mathcal{R},
  \end{equation}
  where $\mathcal{Q} \in \mathfrak{R}^{n_1 \times n_1}$ is orthogonal, $\mathcal{R} \in \mathfrak{R}^{n_1 \times n_2}$ is upper triangular, and we assume $n_1 \geq n_2$.
  \end{definition}

  The Gram-Schmidt orthogonalization procedure used in \cite{kilmer2008third,kilmer2011factorization} may fall victim to the ``catastrophic cancelation" problem \cite{golub2012matrix} where one starts with large values and ends up with small values with large relative errors. Suppose $\bm{A} = [\bm{A}_1,...,\bm{A}_{n_1}]$ and $\bm{Q} = [\bm{Q}_1,...,\bm{Q}_{n_1}]$, the $j$th column $\bm{A}_j$ is successively reduced in length as components in the directions of $\bm{Q}_1,...,\bm{Q}_{j-1}$ are substracted, resulting in a small vector with large relative errors, which will destroy the accuracy of the computed $\bm{Q}_j$, i.e., the newly computed $\bm{Q}_j$ may not be orthogonal to the previous $\bm{Q}_1,...,\bm{Q}_{j-1}$.

\subsection{Householder QR for Matrices}

  It is known that multiplication of a unitary transformation to a matrix or vector inherently preserves the length. The Householder transformation is a kind of unitary transformation, and in the following, we introduce the Householder QR for matrices that is implemented by applying a sequence of Householder transformations.

\subsubsection{\textbf{Householder Transformation}}

  Let $\bm{u} \in \mathbb{R}^{n}$ be a vector of unit length, i.e., $||\bm{u}||=1$. The \textit{Householder transformation} (also called Householder reflections/reflectors) is an $n \times n$ square matrix
  \begin{equation}\label{eq:householder_matrix}
  \bm{P} = \bm{I} - 2\bm{u} \bm{u}^T,
  \end{equation}
  where $\bm{u}$ is called a Householder vector.
  If a vector $\bm{x}$ is multiplied by $\bm{P}$, then it is reflected in the hyperplane span$\{\bm{u}\}^{\perp}$. Note that Householder matrices are unitary.

  \begin{remark}\label{remark:househoulder_reflection}
  (\textbf{Intuitive understanding}) Let $\bm{S} = \text{span}\{\bm{u}\}^{\perp}$ that is a space perpendicular to $\bm{u}$. Imagine the space $\bm{S}$ as a ``mirror", we then the following two interpretations:
  \begin{itemize}
    \item Any vector in $\bm{S}$ (along the imaginary mirror) is not reflected. Let $\bm{z} \in \bm{S}$ be any vector that is perpendicular to $\bm{u}$, we get
        \begin{equation}
        \bm{P} \bm{z} =  (\bm{I} - 2\bm{u} \bm{u}^T)\bm{z} = \bm{z} - 2\bm{u}\underbrace{(\bm{u}^T\bm{z})}_{0} = \bm{z},
        \end{equation}
        which means that $\bm{z}$ is unchanged.
    \item Any vector has the component that is orthogonal to $\bm{S}$ (orthogonal to the minor), then that component reverses in direction. Any vector $\bm{x}$ can be expressed as $\bm{x} = \bm{z} + \bm{u}^T \bm{x} \bm{u}$ where $\bm{z} \in \bm{S}$ (perpendicular to $\bm{u}$) and $\bm{u}^T \bm{x} \bm{u}$ is the component in the direction of $\bm{u}$. We get
        \begin{equation}
        \begin{split}
        \bm{P} \bm{x} =  (\bm{I} - 2\bm{u} \bm{u}^T)(\bm{z} + \bm{u}^T \bm{x} \bm{u}) &= \bm{z} + \bm{u}^T \bm{x} \bm{u} - 2\bm{u}\underbrace{(\bm{u}^T\bm{z})}_{0}-  2\bm{u} \bm{u}^T (\bm{u}^T \bm{x}) \bm{u} \\
        &= \bm{z} + \bm{u}^T \bm{x} \bm{u} - 2 \bm{u}^T \bm{x} \underbrace{(\bm{u}^T \bm{u})}_{1} \bm{u} = \bm{z} - \bm{u}^T \bm{x} \bm{u},
        \end{split}
        \end{equation}
        which means that the component $\bm{u}^T \bm{x} \bm{u}$ has reversed its direction.
  \end{itemize}
  \end{remark}

  \begin{algorithm}[t]
  \caption{Householder QR (\cite{golub2012matrix}, Alg. 5.2.1)}
  \begin{algorithmic}
   \label{alg:complex_QR}
  \STATE \textbf{Input: } $\bm{A} \in \mathbb{R}^{n_1 \times n_2}$
  \FOR{$j = 1~\text{to}~n_2$}
  	\STATE Compute $\bm{u}$ for vector $\bm{x} = \bm{A}(j:n_1,j)$ according to (\ref{eq:householder_transform_matrix_case});\\
    \STATE Construct the $j$th Household matrix $\widehat{\bm{H}}^j = \bm{I}_{n_1 - j + 1} - 2 \bm{u} \bm{u}^H$ where $\bm{I}_{n_1 - j + 1}$ is the identity matrix of size $(n_1 - j + 1) \times (n_1 - j + 1)$.
    \STATE $\bm{A}(j:n_1,j:n_2) \leftarrow  \widehat{\bm{H}}^j  \bm{A}(j:n_1, j:n_2)$;
    \IF{$j < n_1$}
        \STATE $\bm{A}(j+1:n_1,j) = \bm{u}(2:n_1 -j + 1)$;
    \ENDIF
  \ENDFOR
  \end{algorithmic}
\end{algorithm}

  We know from (\ref{eq:householder_matrix}) that Householder matrices are rank-$1$ modifications of the identity matrix and they can be used to zero selected components of a vector. Specifically, suppose we are given a nonzero vector $\bm{x} \in \mathbb{R}^n$ and want $\bm{P} \bm{x}$ to be a scalar multiplication of $\bm{e}_1 = \bm{I}_1$. Note that
  \begin{equation}
  \bm{P}\bm{x} = \left(\bm{I} - \frac{2\bm{u}\bm{u}^T}{\bm{u}^T \bm{u}} \right) \bm{x} = \bm{x} -  \frac{2\bm{u}^T\bm{x}}{\bm{u}^T \bm{u}}\bm{u},~~\text{and}~~\bm{P}\bm{x} \in \text{span}\{\bm{e}_1\}
  \end{equation}
  imply that $\bm{u} \in \text{span}\{\bm{x},\bm{e}_1\}$, i.e., $\bm{u}$ can be expressed as  $\bm{u} = \bm{x} + a \bm{e}_1$.
  Setting $a = \pm ||\bm{x}||_2$ \cite{golub2012matrix}, we get
  \begin{equation}\label{eq:householder_transform_matrix_case}
  \bm{u} = \bm{x} \pm ||\bm{x}||_2\bm{e}_1~~\Longrightarrow~~\bm{P}\bm{x} = \left(\bm{I} - \frac{2\bm{u}\bm{u}^T}{\bm{u}^T \bm{u}} \right) \bm{x} = \mp ||\bm{x}||_2 \bm{e}_1.
  \end{equation}

  Given a matrix $\bm{A} \in \mathbb{R}^{n_1 \times n_2}$ ($n_1 \geq n_2$), Alg. \ref{alg:complex_QR} (\cite{golub2012matrix}, Alg. 5.2.1) finds Householder matrices $\bm{H}_1,...,\bm{H}_{n_2}$. Let $\bm{Q} = \bm{H}_1\cdots \bm{H}_{n_2}$, than $\bm{Q}^T \bm{A} = \bm{R}$ is upper triangular. In the pseudocode of Alg. \ref{alg:complex_QR}, the upper triangular part of $\bm{A}$ is overwritten by the upper triangular par of $\bm{R}$ and components $j+1:m$ of the $j$th Householder vector are stored in $\bm{A}(j+1:m,j),~j< m$.

%
%
%

\begin{remark}\label{remark:reflection_direction}
Considering (\ref{eq:householder_transform_matrix_case}), for a given vector $\bm{x}$, there exists a Householder vector $\bm{u}$ such that $\bm{x}$ is reflected into another vector $\bm{y}$, i.e., $(\bm{I} - 2\bm{u}\bm{u}^T) \bm{x} = \bm{y}$ with $||\bm{x}||_2 = ||\bm{y}||_2$. Note that the vector $\bm{v} = \bm{x} - \bm{y}$ is perpendicular to the space $\bm{S} = \text{span}\{\bm{u}\}^{\perp}$, implying that $\bm{u}$ equals a unit vector in the direction $\bm{v}$, i.e., $\bm{u} = \bm{v} / ||\bm{v}||_2$.
\end{remark}

\subsection{Householder $\mathcal{L}$-QR}

  In a tensor subspace in Definition \ref{def:tensor_column_row_subspaces}, we want to transform (or reflect) a given tensor-column vector $\bm{x} \in \mathfrak{R}^{n}$ into another tensor-column vector $\bm{y} \in \mathfrak{R}^{n}$ where $||\bm{x}||_F = ||\bm{y}||_F$.  Note that the signs $\pm$ defined for real scalars does not work for tensor-scalars in Definition \ref{def:tensor_scalar}, thus (\ref{eq:householder_transform_matrix_case}) is not directly available for us to reverse the direction as in Remark \ref{remark:househoulder_reflection}.
%

\subsubsection{\textbf{Householder Transformation}}

  Let $\bm{u} \in \mathfrak{R}^{n}$ be a nonzero vector of tensor-scalars. The \textit{Householder transformation} is an $n \times n$ square tensor $\mathcal{P} \in \mathfrak{R}^{n \times n}$ of the following form
  \begin{equation}\label{eq:householder_tensor}
  \mathcal{P} = \mathcal{I} - 2 (\bm{u}^H \bullet \bm{u})^{-1} \bullet \bm{u} \bullet \bm{u}^H.
  \end{equation}
  If a vector $\bm{x} \in \mathfrak{R}^n$ is multiplied by $\mathcal{P}$, then it is reflected in the hyperplane t-span$\{\bm{u}\}^{\perp}$.

  Given a vector $\bm{x}$, the key is to find the ``right" Household vector $\bm{u}$.

  For both $\bm{x}$ and $\bm{u}$, partition them into two parts: the first tensor-scalar and the rest. Let $\bm{x}' = \bm{x}_{[2:n]}$ and $\bm{u}' = \bm{u}_{[2:n]}$, we want to find a House vector $\bm{u}$ such that
  \begin{equation}\label{eq:householder_L_QR_one_vector}
  \begin{split}
   &\bm{x} = \left[
   \begin{array}{c}
   \bm{x}_1\\
   \bm{x}'
   \end{array}
   \right],~~\bm{u} = \left[
   \begin{array}{c}
   \bm{u}_1\\
   \bm{u}'
   \end{array}
   \right],\\
   &\left(\mathcal{I} - \frac{2}{\bm{u}^H \bullet \bm{u}} \bullet \left[
   \begin{array}{c}
   \bm{u}_1\\
   \bm{u}'
   \end{array}
   \right] \bullet \left[
   \begin{array}{c}
   \bm{u}_1\\
   \bm{u}'
   \end{array}
   \right]^H
    \right) \left[
   \begin{array}{c}
   \bm{x}_1\\
   \bm{x}'
   \end{array}
   \right] = \left[
   \begin{array}{c}
   \angle \bm{x}_1 \bullet ||\bm{x}||_F\\
   \bm{0}'
   \end{array}
   \right],
   \end{split}
  \end{equation}
  where $\bm{0}'$ denotes a vector with $n-1$ zero tensor-scalars.

  As pointed out in Remark \ref{remark:reflection_direction}, we know that the direction of the Householder vector $\bm{u}$ is given by the direction by
  \begin{equation}
  \bm{v} = \left[
   \begin{array}{c}
   \bm{x}_1 - \angle \bm{x}_1 \bullet ||\bm{x}||_F\\
   \bm{x}'
   \end{array}
   \right].
  \end{equation}
  Then, we normalize this vector to make the first entry equals to ``e":
  \begin{equation}
  \bm{u} = \frac{\bm{v}}{||\bm{v}||_F} = \frac{1}{\bm{x}_1 - \angle \bm{x}_1 \bullet ||\bm{x}||_F} \bullet
  \left[
   \begin{array}{c}
   \bm{x}_1 - \angle \bm{x}_1 \bullet ||\bm{x}||_F\\
   \bm{x}'
   \end{array}
   \right] = \left[
   \begin{array}{c}
   e\\
   \bm{x}'/\bm{v}_1
   \end{array}
   \right],
  \end{equation}
  where $\bm{v}_1 = \bm{x}_1 - \angle \bm{x}_1 \bullet ||\bm{x}||_F$.

  \begin{algorithm}[t]
  \caption{Householder $\mathcal{L}$-QR}
  \begin{algorithmic}
   \label{alg:Householder_L_QR}
  \STATE \textbf{Input: } $\mathcal{A} \in \mathfrak{R}^{n_1 \times n_2}$
  \STATE Initialize $\mathcal{Q}^0 = \mathcal{I}_{n_1}$  where $\mathcal{I}_{n_1}$ is the identity matrix of size $n_1 \times n_1$.
  \FOR{$j = 1~\text{to}~n_2$}
  	\STATE $\bm{x} = \mathcal{A}(j:n_1,j)$;\\
    \STATE $\bm{u} = \bm{x} -  \angle \bm{x}_1 \bullet ||\bm{x}||_F \bullet \bm{e}_1$ where $\bm{x}_1 = \angle \bm{x}_1 \bullet \text{abs}(\bm{x})$;\\
    \STATE $\widehat{\mathcal{H}}^j = \mathcal{I}_{n_1 - j + 1} - 2(\bm{u}^H \ \bullet \bm{u})^{-1} \bullet \bm{u} \bullet \bm{u}^H$;\\
    \STATE $\mathcal{A}(j:n_1,j:n_2) = \widehat{\mathcal{H}}^j \bullet \mathcal{A}(j:n_1, j:n_2)$;
    \IF{$j \geq 2$}
    \STATE $\mathcal{Q}^{j} \leftarrow  \text{diag}(\mathcal{I}_{j-1},\widehat{\mathcal{H}}^j) \bullet \mathcal{Q}^{j-1}$;
    \ELSE
    \STATE $\mathcal{Q}^{j} \leftarrow \mathcal{Q}^{j-1} \bullet \widehat{\mathcal{H}}^j$.
    \ENDIF
  \ENDFOR
  \end{algorithmic}
\end{algorithm}

  \begin{theorem}\label{theorem:householder_L_QR}
  Alg. \ref{alg:Householder_L_QR} outputs an $\mathcal{L}$-QR as in Definition \ref{def:tensor_L_QR}. If $\mathcal{A} = \mathcal{Q} \bullet \mathcal{R}$ has full tensor-column rank, suppose $\mathcal{A} = [\mathcal{A}_1,...,\mathcal{A}_{n_1}]$ and $\mathcal{Q} = [\mathcal{Q}_1,...,\mathcal{Q}_{n_1}]$, then
  \begin{equation}
  \text{t-span}(\mathcal{A}_1,...,\mathcal{A}_{k}) = \text{t-span}(\mathcal{Q}_1,...,\mathcal{Q}_{k}),~~k \in [n_1].
  \end{equation}
  \end{theorem}

  \begin{proof}
  We use the induction method to prove that Alg. \ref{alg:Householder_L_QR} outputs a structured $\mathcal{L}$-QR as in Definition \ref{def:tensor_L_QR}, such that $\mathcal{Q} \in \mathfrak{R}^{n_1 \times n_1}$ is orthogonal and $\mathcal{R} \in \mathfrak{R}^{n_1 \times n_2}$ is upper triangular. In the first iteration of Alg. \ref{alg:Householder_L_QR}, $\bm{x} = \mathcal{A}_1$, we know that
  \begin{equation}
  \begin{split}
  \widehat{\mathcal{H}}^1
  \left[
   \begin{array}{c}
   \bm{x}_1\\
   \bm{x}'
   \end{array}
   \right] = \left[
   \begin{array}{c}
   \angle \bm{x}_1 \bullet ||\bm{x}||_F\\
   \bm{0}'
   \end{array}
   \right],
  \end{split}
  \end{equation}
  according to (\ref{eq:householder_L_QR_one_vector}). Therefore, we have
  \begin{equation}
  \begin{split}
  \widehat{\mathcal{Q}}^1 &= \widehat{\mathcal{Q}}^0 \bullet \widehat{\mathcal{H}}^1 = \widehat{\mathcal{H}}^1, \\
  \widehat{\mathcal{Q}}^1 \bullet
  \left[
   \begin{array}{c}
   \mathcal{A}_{11}\\
   \mathcal{A}_1'
   \end{array}, \mathcal{A}_2, \cdots, \mathcal{A}_{n_2}
   \right] &= \left[
   \begin{array}{c}
   \angle \bm{x}_1 \bullet ||\bm{x}||_F\\
   \bm{0}'
   \end{array}, \widehat{\mathcal{H}}^1 \bullet \mathcal{A}_2, \cdots, \widehat{\mathcal{H}}^1 \bullet \mathcal{A}_{n_2}
   \right],
  \end{split}
  \end{equation}
  Let us assume that the triangulation (\ref{def:tensor_L_QR}) is correct in the $j$th iteration, i.e.,
  \begin{equation}
  \begin{split}
  \widehat{\mathcal{Q}}^j \bullet
  \left[\mathcal{A}_1, \mathcal{A}_2, \cdots, \mathcal{A}_{j}
   \right] &= \mathcal{R}_{[1:j]} = \left[\mathcal{R}_1, \mathcal{R}_2, \cdots, \mathcal{R}_{j}
   \right],
  \end{split}
  \end{equation}
  where $\mathcal{R}_{[1:j]}$ is upper triangular.

  Next, we verify that in the $(j+1)$th iteration,
  \begin{equation}
  \begin{split}
  \widehat{\mathcal{Q}}^{j+1} \bullet
  \left[\mathcal{A}_1, \mathcal{A}_2, \cdots, \mathcal{A}_{j+1}
   \right] &= \mathcal{R}_{[1:j+1]} = \left[\mathcal{R}_1, \mathcal{R}_2, \cdots, \mathcal{R}_{j+1}
   \right],
  \end{split}
  \end{equation}
  where $\mathcal{R}_{[1:j+1]}$ is upper triangular.
  Since   $\mathcal{Q}^{j+1} = \text{diag}(\mathcal{I}_{j},\widehat{\mathcal{H}}^{j+1}) \bullet  \mathcal{Q}^{j}$ and $\mathcal{R}(j+1:n_1,j+1:n_2) = \widehat{\mathcal{H}}^{j+1} \bullet \mathcal{A}(j+1:n_1, j+1:n_2)$,
  \begin{equation}
  \begin{split}
  \widehat{\mathcal{Q}}^{j+1} \bullet
  \mathcal{A}_{[1:j+1]} &=
  \left[
   \begin{array}{cc}
   \mathcal{I}_j & \\
    & \widehat{\mathcal{H}}^{j+1}
   \end{array}
   \right] \bullet (\widehat{\mathcal{Q}}^j \bullet
  \mathcal{A}_{[1:j+1]}) =  \left[
   \begin{array}{cc}
   \mathcal{I}_j & \\
    & \widehat{\mathcal{H}}^{j+1}
   \end{array}
   \right] \bullet \left[\mathcal{R}_{[1:j]}, \widehat{\mathcal{Q}}^j \bullet \mathcal{A}_{j+1} \right]\\
   & = \left[\mathcal{R}_{[1:j]}, \left[
   \begin{array}{cc}
   \mathcal{I}_j & \\
    & \widehat{\mathcal{H}}^{j+1}
   \end{array}
   \right] \bullet \left[
   \begin{array}{c}
   (\widehat{\mathcal{Q}}^j \bullet \mathcal{A}_{j+1})_{1:j-1} \\
   (\widehat{\mathcal{Q}}^j \bullet \mathcal{A}_{j+1})_{j:n_1}
   \end{array}
   \right]\right].
  \end{split}
  \end{equation}
  According to (\ref{eq:householder_L_QR_one_vector}), we know that the $j+1$th element of $\widehat{\mathcal{H}}^{j+1} \bullet (\widehat{\mathcal{Q}}^j \bullet \mathcal{A}_{j+1})_{j:n_1}$ is nonzero and the $j+2:n_1$th elements are zero. Therefore, $\mathcal{R}_{[1: j+1]}$ is upper triangular.

  Comparing the $k$th tensor-columns in $\mathcal{A} = \mathcal{Q} \bullet \mathcal{R}$ we get
  \begin{equation}
  \mathcal{A}_k = \sum\limits_{i=[k]} \mathcal{R}_{ik} \bullet \mathcal{Q}_i \in \text{t-span}\{\mathcal{Q}_1,...,\mathcal{Q}_k\}.
  \end{equation}
  Thus, $\text{t-span}(\mathcal{A}_1,...,\mathcal{A}_{k}) \subset \text{t-span}(\mathcal{Q}_1,...,\mathcal{Q}_{k})$.
  Since $\mathcal{A} = \mathcal{Q} \bullet \mathcal{R}$ has full tensor-column rank, $\text{rank}(\mathcal{A}) = n_2$, it follows that $\text{t-span}(\mathcal{A}_1,...,\mathcal{A}_{k})$ has dimension $k$ and so must equal $\text{t-span}(\mathcal{Q}_1,...,\mathcal{Q}_{k})$. The proof is completed.
  \end{proof}

\section{Performance Evaluation}
\label{sec:performance_evaluation}

  We apply the proposed framework to two examplar applications: video compression and one-shot face recognition. For video compression, we select two realistic and representative scenarios: an online NBA basketball video \cite{basketball_video} and a drone video of the Central Park in autumn \cite{central_park}. For one-shot face recognition, we use the Weizmann face database \cite{weizmann}.

  Our experiment platform is a Matlab IDE installed on a server with Linux operating system. The parameters of the server are as follows: Intel$\circledR$ Xeon$^{\circledR}$ Processor E5-2650 v3, $2.3$ GHz clock speed, $2$ CPU each having $10$ physical cores (virtually maximum $40$ threads), $25$ MB cache, and $64$ GB memory.

\subsection{Video Compression}

  \begin{figure}[t]\centering
  \includegraphics[width=0.49\textwidth]{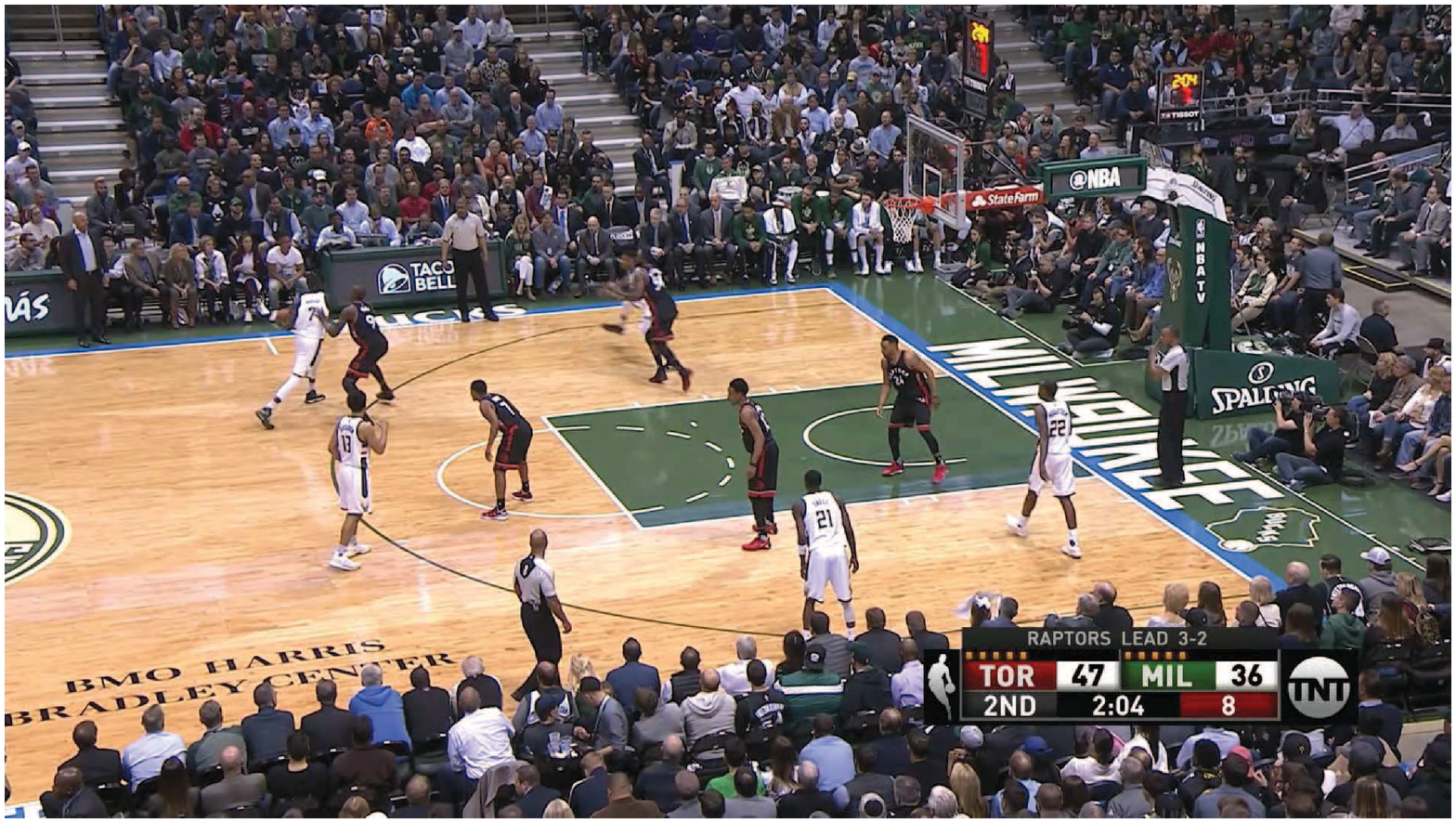}
  \includegraphics[width=0.49\textwidth]{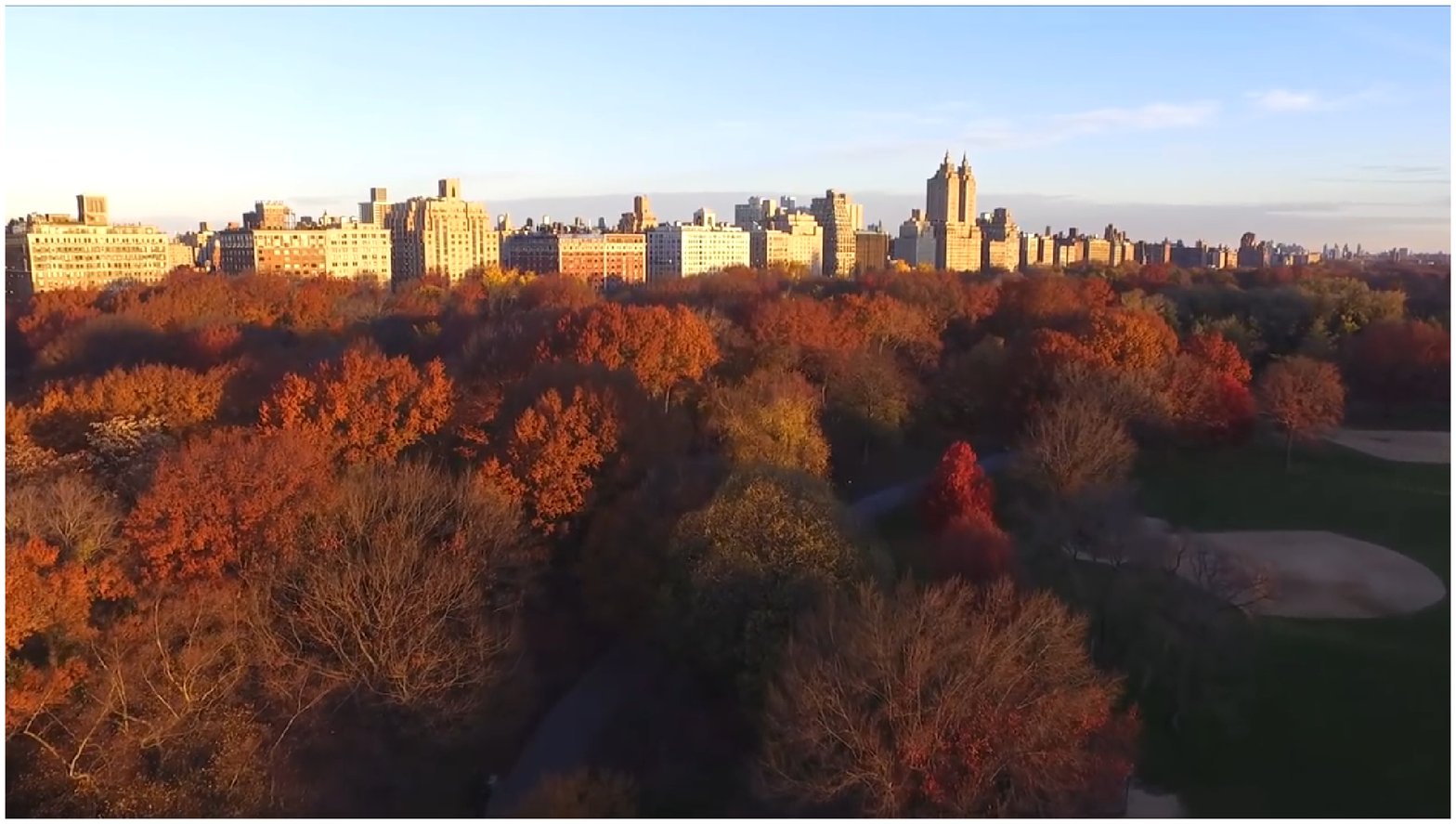}
  \caption{Overview of the NBA basketball video (the $30$th frame) and the drone video of Central Park in autumn (the $8,100$th frame).}\label{fig:video_overview}
  \end{figure}

   We ultilize two realistic video datasets, an online NBA basketball video \cite{basketball_video} (NBA video) and a drone video of the Central Park in autumn \cite{central_park} (CPark video), as shown in Fig. \ref{fig:video_overview}. Both videos have $29.97$ frames per second, and are in RGB$24$ format.
   The NBA video lasts for $17$ seconds with $535$ frames available, and has size $1,080 \times 1,920 \times 3 \times 535$. The CPark video lasts for $5$ minutes with $8,980$ frames available, and has size $720 \times 1,280 \times 3 \times 8,980$. To avoid memory outage, we use the first $120$ frames ($\approx 4$ seconds) of the NBA video and the $8,001$th to the $8,300$th frames ($\approx 10$ seconds) of the CPark video.

   First, we describe the compression methods, comparing with the truncated SVD based approach. The compression performance is measured by the ratio of the total numbers of entries in the SVD factors to the total number of entries in the original tensor.

   \begin{itemize}
     \item SVD (truncated SVD) \cite{klema1980singular,shi1999image}: For a matrix SVD $\bm{A} = \bm{U} \bm{S} \bm{V}^T \in \mathbb{R}^{n_1 \times n_2}$, the rank-$r$ approximation is $\bm{A}_r = \bm{U}_r \bm{S}_r \bm{V}_r^T$ with $r \leq \min(n_1, n_2)$, where $\bm{S}_r$ is an $r \times r$ diagonal matrix, $\bm{U}_r$ consists of the first $r$ columns of $U$, and $\bm{V}_r^T$ consists of the first $r$ rows of $\bm{V}^T$. The total number of entries in $\bm{U}_r$, $\bm{S}_r$, and $\bm{V}_r^T$ equals to $(n_1 + n_2 + 1)r$. Extending this approach to a fourth-order tensor $\mathcal{A} \in \mathbb{R}^{n_1 \times n_2 \times n_3 \times n_4}$ as follows: we use $n_1 \times n_2$ to denote the resolution, set $n_3 =3$ and $n_4$ to be the number of frames, and we perform SVD on each $n_1 \times n_2$ matrix. Then the compression ratio of rank-$r_1$ approximation is
         \begin{equation}
         \text{ratio}_{\text{SVD}} = \frac{(n_1 + n_2 + 1)r_1 n_3n_4}{n_1n_2n_3n_4}=\frac{(n_1 + n_2 + 1)r_1}{n_1n_2},
         \end{equation}
         where $1 \leq r_1 \leq \min(n_1, n_2)$. Note that $r_1$ is selected to be the maximum of the ranks of those $n_3n_4$ matrices, which indicates that collectively compress the tensor may have lower compression ratio than dealing each matrix separately.
     \item $\mathcal{L}$-SVD in Alg. \ref{alg:L_SVD}: we carry out the compression in the transform-domains. First, we set $n_1$ to be the frame-rows, $n_2$ to be the number of frames, $n_3 =3$ and $n_4$ to be the frame-columns. E.g., for the NBA video of size $1,080 \times 1,920 \times 3 \times 120$, its tensor representation $\mathcal{A} \in \mathbb{R}^{1080 \times 120 \times 3 \times 1920}$.    Secondly, we compute $\mathcal{L}$-SVD as in Alg. \ref{alg:L_SVD} and $\widetilde{\mathcal{U}}^p,~\widetilde{\mathcal{S}}^p$ and $\widetilde{\mathcal{V}}^p$ for $p \in [n_3n_4]$. It is know that $\widetilde{\mathcal{S}}^p$ is diagonal, so the total number of diagonal entries of $\widetilde{\mathcal{S}}^p,~p \in [n_3n_4]$ is $n_3n_4\cdot \min(n_1, n_2)$. Thirdly, we choose an integer $r_2,~1 \leq r_2 \leq n_3n_4\cdot \min(n_1, n_2)$, keep the $r_2$ largest diagonal entries of all $\widetilde{\mathcal{S}}^p,~p \in [n_3n_4]$ and set the rest to be $0$. If $\widetilde{\mathcal{S}}^p(i,i)$ is set to be $0$, then let the corresponding columns $\widetilde{\mathcal{U}}^p(:,i)$ and rows $\widetilde{\mathcal{V}}^p(i,:)$ also be $0$. Let us call the resulting tensors $\mathcal{U}_{r_2}$, $\mathcal{S}_{r_2}$ and $\mathcal{V}_{r_2}^{H}$, and the approximation $\mathcal{A}_{r_2}$. Then the compression ratio of $\mathcal{L}$-SVD approximation is
         \begin{equation}
         \text{ratio}_{\text{SVD}} = \frac{(n_1 + n_2 + 1)r_2}{n_1n_2n_3n_4},
         \end{equation}
         where $1 \leq r_2 \leq n_3n_4\cdot \min(n_1, n_2)$.
   \end{itemize}

   Let $\mathcal{L}$ be the two-dimensional Fourier transform, we get t-SVD \cite{martin2013order} (similar to \cite{Shuchin2014CVPR} that compressed grey videos, being third-order tensors). We also test when the transform $\mathcal{L}$ is set to be a discrete cosine transform (dct-SVD) and a Daubechies-4 Discrete Wavelet Transform (dwt-SVD) \cite{daubechies1992ten}, respectively. For the above compression method, we measure the approximation performance via the relative square error (RSE) that is defined in dB, i.e., $\text{RSE} = 20 \log_{10}(||\mathcal{A} - \mathcal{A}_r ||_F / ||\mathcal{A}||_F)$. Note that a $3$dB gain corresponds to $\frac{\sqrt{2}}{2}=0.707$, i.e., $\text{RSE}_{\text{new}} = 0.707 \cdot \text{RSE}_{\text{old}}$, while a $5$dB gain corresponds to $0.5623$ and a $10$dB gain corresponds to $0.316$, respectively.

  Fig. \ref{fig:video_compression} shows the compression results for the two videos. All tensor-based compression methods have lower RSE, since the tensor representation can exploit the inter-frame correlations. For the NBA video, dct-SVD achieves about $3$dB gains over tSVD while dwt-SVD achieves $3\sim 10$dB gains. This observation suggests that the human movements in videos are better captured by the discrete cosine domain representation and the discrete wavelet domain representation.

  For the CPark video, all compression methods have lower RSE errors, while the performance improvements of exploiting transforms are less, e.g., dwt-SVD achieves $3\sim 5$dB gains over tSVD. dct-SVD has $1 \sim 3$dB gains over matrix SVD for cases with compression ratio less than $0.6$, while tSVD is only slightly better than matrix SVD. However, for  cases with compression ratio bigger than $0.6$, dct-SVD and tSVD behave almost the same with matrix SVD. The possible reasons would be: 1) the CPark video captures an overview of a park, being much bigger than the basketball field, therefore, each frame is strongly compressible already and there is less improvement space for using a better transform; and 2) the NBA video has more stationary background while the players' activities can be better captured by transforms.

  Fig. \ref{fig:video_compression_runningtime} compares the running time of tSVD, dct-SVD and dwt-SVD, while we do not include the running time of SVD as it is unfair. For both videos, the running time increases as the compression ratio increases. Comparing dct-SVD and tSVD, it verifies the our intuition that the discrete cosine transform involves about half amount of computations taken by the Fourier transform. As expected, dwt-SVD requires much less amount of computations.

  \begin{figure}[t]\centering
  \includegraphics[width=0.49\textwidth]{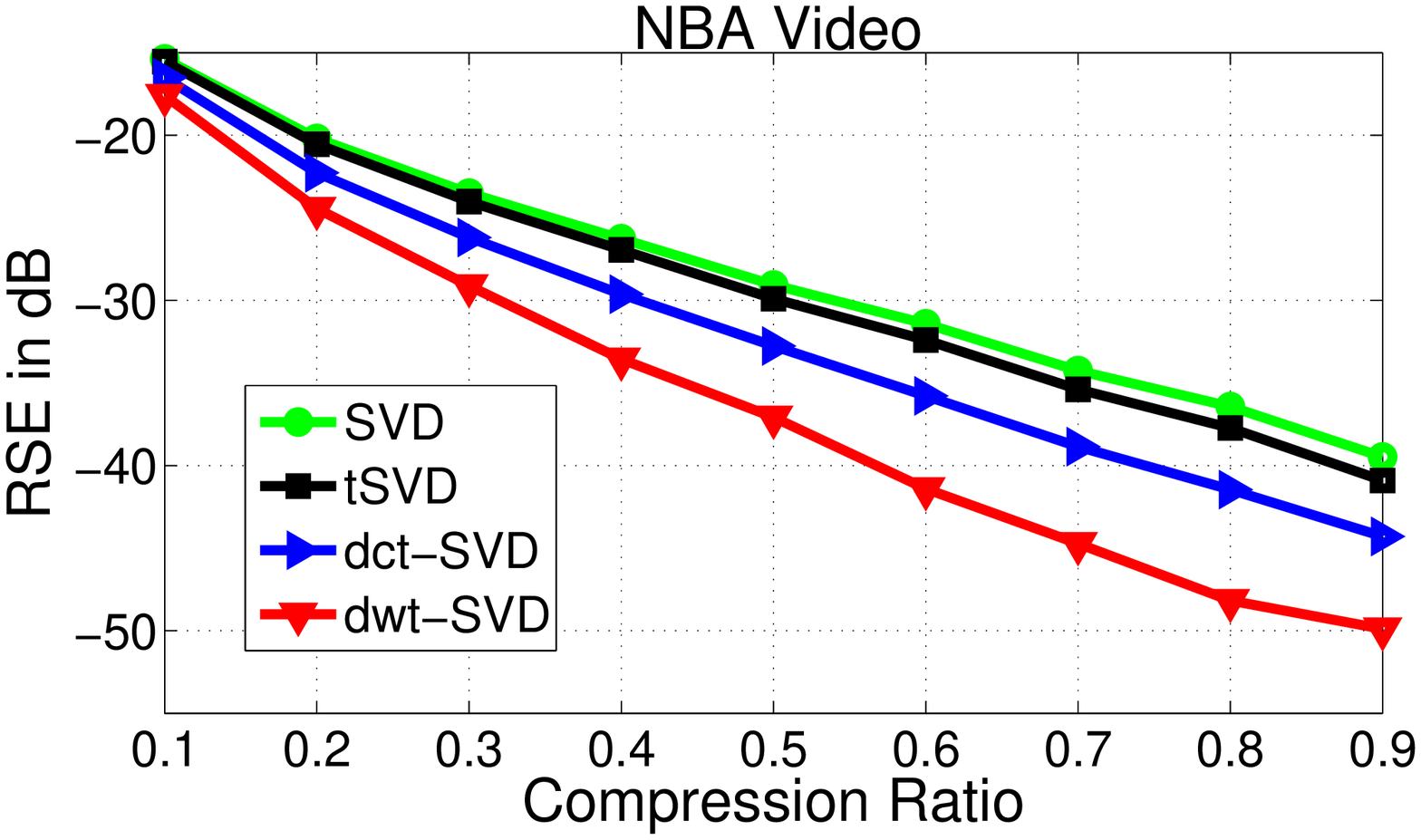}
  \includegraphics[width=0.49\textwidth]{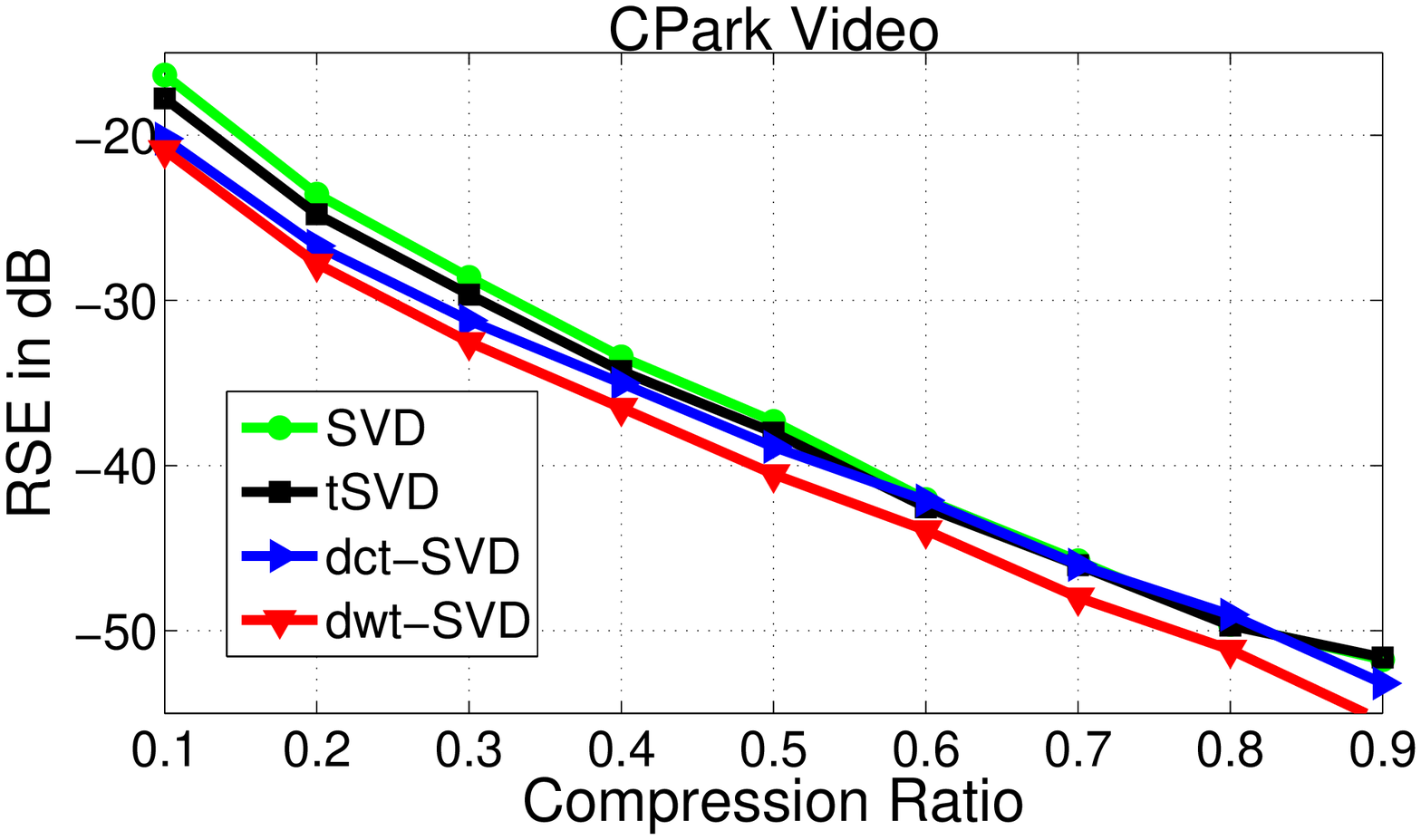}
  \caption{Comparison of compression performance: SVD, t-SVD, dct-SVD, and dwt-SVD.}\label{fig:video_compression}
  \end{figure}

  \begin{figure}[t]\centering
  \includegraphics[width=0.49\textwidth]{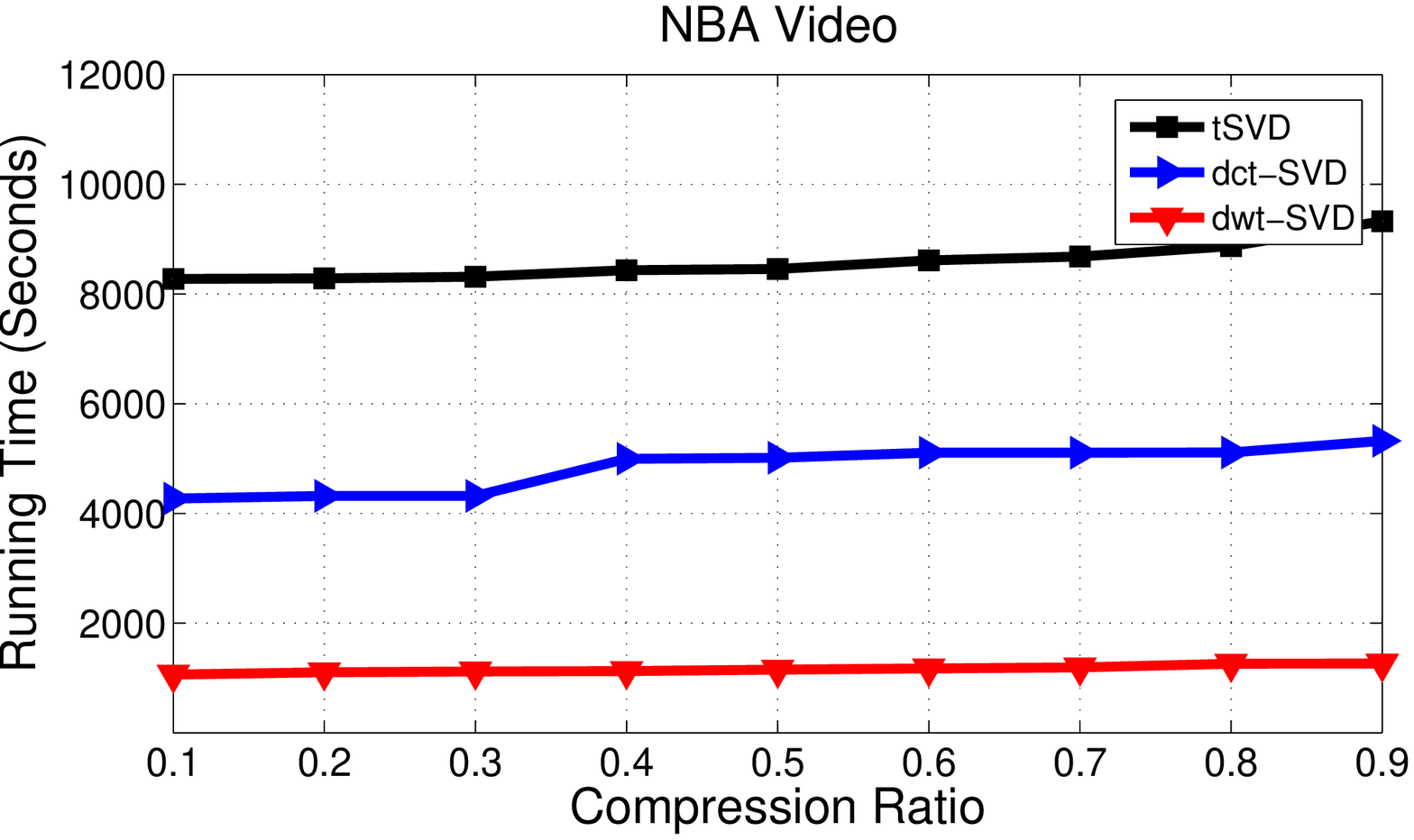}
  \includegraphics[width=0.49\textwidth]{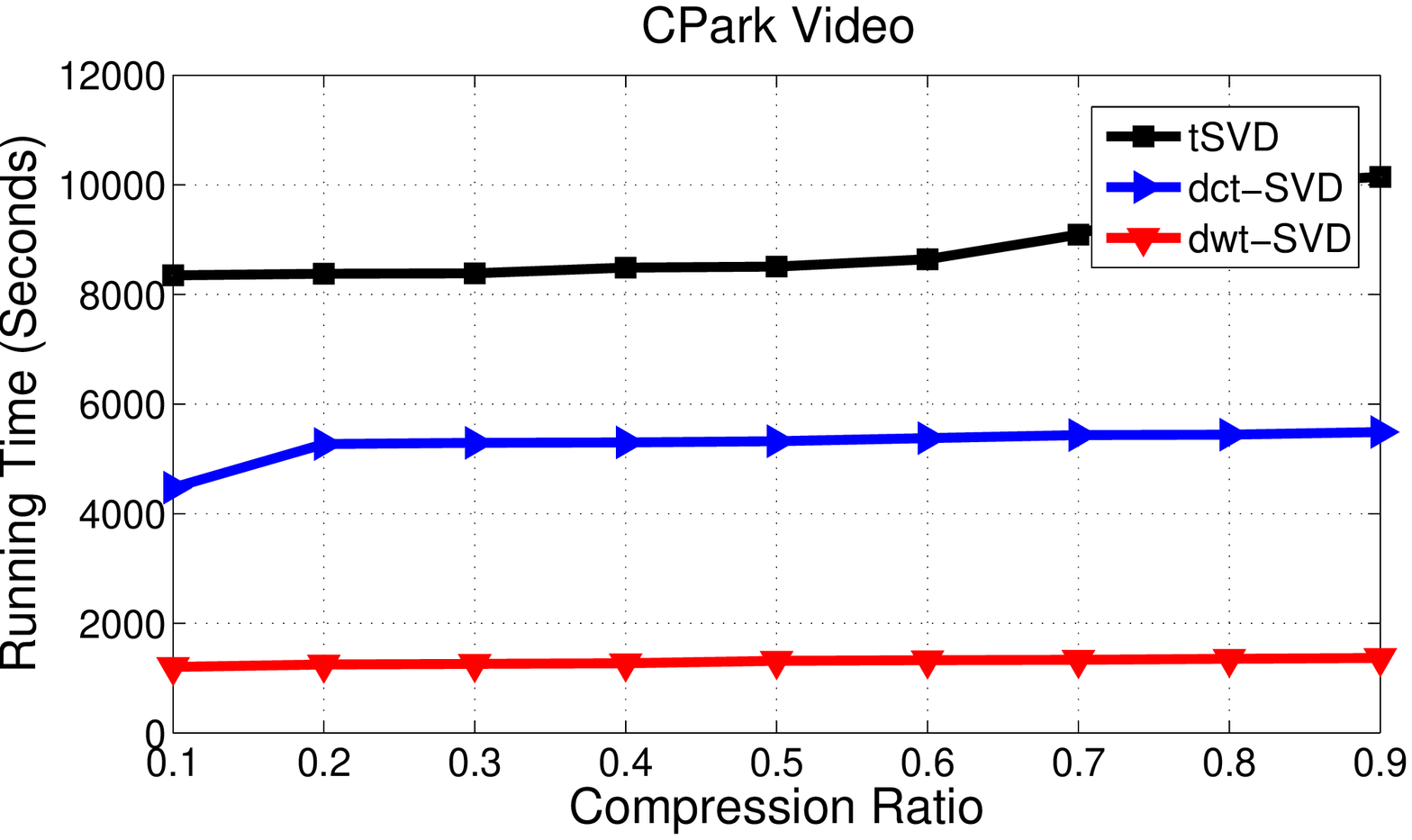}
  \caption{Comparison of the running time: t-SVD, dct-SVD, and dwt-SVD.}\label{fig:video_compression_runningtime}
  \end{figure}

\subsection{One-shot Face Recognition}

  We apply the $\mathcal{L}$-based tSVD to one-shot face recognition. In one-shot face recognition, the training data set has limited number of images of each person and we want to recognize a set of images of an unlabeled person. The one-shot face recognition algorithm is given in Alg. \ref{alg:video_face_recognition} which will be described in detailed in the following.

   Let $\mathcal{F}_1,..., \mathcal{F}_j,...,\mathcal{F}_{n_2} \in \mathbb{R}^{n_1 \times 1 \times n_3 \times n_4},~j \in [n_2]$ be a collection of $n_2$ videos, where we have a video over $n_4$ time slots for the $j$th person, namely, $n_4$ frames of size $n_1 \times n_3$. We use a tensor $\mathcal{A} \in \mathbb{R}^{n_1 \times n_2 \times n_3 \times n_4}$ to denote the mean substracted frames, i.e., $\mathcal{A}_j = \mathcal{F}_j - \Psi$ where
   \begin{equation}
   \Psi = \frac{1}{n_2} \sum\limits_{j \in [n_2]} \mathcal{F}_j.
   \end{equation}
   The covariance tensor $\mathcal{C}$ of $\mathcal{A}$ is given by $\mathcal{C} = \mathcal{A} \bullet \mathcal{A}^H$. Computing the $\mathcal{L}$-SVD $\mathcal{A} = \mathcal{U} \bullet \mathcal{S} \bullet \mathcal{V}^T$, then we know that $\mathcal{C} = \mathcal{U} \bullet \mathcal{S} \bullet \mathcal{S}^H \bullet \mathcal{U}^H$ where $\mathcal{U}$ is regarded as the ``video spaces". Then we project $\mathcal{A}$ on to the video spaces $\mathcal{U}$ as $\mathcal{G} = \mathcal{U}^H \bullet \mathcal{A}$. Note that by choosing some smaller $r$ so that $\mathcal{U}_{[r]} = \mathcal{U}(:,1:r)^H$, we can accelerate the computation.

   Let $\mathcal{T} \in \mathbb{R}^{n_1 \times 1 \times n_3 \times n_4},~j \in [n_2]$ denote a new video (or $n_4$ images) where each frame is size $n_1 \times n_3$, to be tested. We firs substract $\Psi$ and then do the projection, i.e., $\bm{c} = \mathcal{U}^H \bullet (\mathcal{T} - \Psi)$. Comparing the coefficients $\bm{c}$ with $\mathcal{G}$, we determine the classification to be the $j$th person with minimum distance (in $\ell_1$-norm).

  \begin{algorithm}[t]
  \caption{One-shot Face Recognition based on $\mathcal{L}$-SVD}
  \begin{algorithmic}
   \label{alg:video_face_recognition}
  \STATE \textbf{Input:} training video set $\mathcal{F}_1,..., \mathcal{F}_j,...,\mathcal{F}_{n_2}$, testing video $\mathcal{T}$, parameter $r$.
  \STATE Compute the mean video $\Psi = \frac{1}{n_2} \sum\limits_{j \in [n_2]} \mathcal{F}_j$;\\
  \FOR{$j = 1~\text{to}~n_2$}
  	\STATE $\mathcal{A}_j = \mathcal{F}_j - \Psi$;\\
  \ENDFOR
  \STATE $[\mathcal{U},~\mathcal{S},~\mathcal{V}] = \text{$\mathcal{L}$-SVD}(\mathcal{A})$ by Alg. \ref{alg:L_SVD};
  \STATE $\mathcal{G} = \mathcal{U}_{[r]}^H \bullet \mathcal{A}$;
  \STATE $\mathcal{T}' = \mathcal{T} - \Psi$;
  \STATE $\bm{c} = \mathcal{U}_{[r]} \bullet \mathcal{T}'$;
  \STATE $\widehat{j} \leftarrow \argmin_{j}~||\bm{c} - \mathcal{G}_j||_1$
  \end{algorithmic}
  \end{algorithm}

  We use the Weizmann face database \cite{weizmann}, which contains $28$ male persons in five viewpoints, three illuminations, and three expressions. Each image is size $512 \times 352 \times 3$. For computer memory reasons we reduced the resolution of the images to size $128 \times 88$. In the experiments, the training set was an fourth-order tensor consisting of third-order blocks for each of the $28$ peoples in the Weizmann database. The testing set was a third-order tensor of images over the various expressions or illuminations (which can also be thought of as movement through time). The baseline algorithm is the convolutional neural networks (CNN), where we adopt the implementation already included in the MATLAB deep learning toolbox.

  \begin{figure}[t]\centering
  \includegraphics[width=0.52\textwidth]{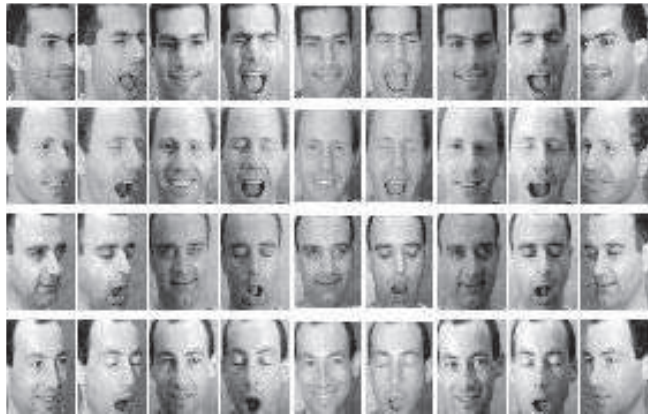}
  \caption{Overview: sample faces from the Weizmann face set.}\label{fig:face_overview}
  \end{figure}

  The recognition rate is defined to be the number of cases where $\widehat{j} = j$ to the total number of trials. All training and testing sets used all $28$ people, and the results are listed out in Table I. The recognition rate is averaged over the feature not listed in columns $1$, e.g., the first testing set rate of $70.3\%$ is averaged over three illuminations. The combinations in Table I create many cross-comparisons.

  As shown in Table I, we saw that CNN's recognition rate is not very satisfying for one-shot face recognition. Compared with the state-of-the-art high accuracy (over $90\%$ \cite{lecun2015deep}), the key difference is that there are only limited number of available images for training. tSVD's recognition rates are comparable to those of CNN, while we observe an improvement $5\% \sim 10\%$ for dct-SVD and  $13\% \sim 23\%$ dwt-SVD. Note that in the case ``exp. 1-3, view 2", dct-SVD's recognition rate is $4\%$ lower than that of CNN.

   \begin{table*}[t] \label{recognition_rate_table}\centering
   \setlength{\abovedisplayskip}{4pt}
   \caption{One-shot face recognition for the Weizmann face database.}
   \begin{tabular}{|c|c|c|c|c|c|}
   \hline
   Testing set & Training set & tSVD & dct-SVD & dwt-SVD & CNN  \\
   \hline\hline
   views 2-4, exp. 1  & exp. 1-3, ill. 1, view 3 & $70.3\%$ & $76.1\%$  & $88.2\%$ & $71.2\%$\\
   \hline
   views 2-4, exp. 2  & exp. 1-3, ill. 1, view 3 & $78.1\%$ & $80.5\%$ & $85.6\%$ & $72.5\%$\\
   \hline
   views 2-4, exp. 3  & exp. 1-3, ill. 1, view 3 & $75.5\%$ & $75.7\%$ & $87.9\%$ & $67.3\%$\\
   \hline\hline
   exp. 1-3, view 1  & views 2-4, ill. 1, exp. 1 & $53.4\%$ & $67.8\%$ & $77.3\%$ & $59.8\%$\\
   \hline
   exp. 1-3, view 2  & views 2-4, ill. 1, exp. 1 & $71.7\%$ & $69.3\%$ & $89.1\%$ & $73.3\%$\\
   \hline
   exp. 1-3, view 3  & views 2-4, ill. 1, exp. 1 & $70.2\%$ & $73.4\%$ & $91.6\%$ & $68.7\%$\\
   \hline\hline
   ill. 1-3, view 1  & views 2-4, ill. 1, exp. 1 & $61.4\%$ & $75.1\%$ & $87.4\%$& $68.5\%$\\
   \hline
   ill. 1-3, view 2  & views 2-4, ill. 1, exp. 1 & $79.7\%$ & $82.8\%$ &$89.1\%$& $75.6\%$\\
   \hline\hline
   ill. 1-3, view 1  & views 1,3,5, ill. 1, exp. 1 & $71.2\%$ & $79.4\%$ & $92.3\%$ & $77.3\%$\\
   \hline
   ill. 1-3, view 2  & views 1,3,5, ill. 1, exp. 1 & $74.6\%$ & $83.1\%$ & $87.3\%$ & $81.3\%$\\
   \hline
   \end{tabular}\vspace{-6pt}
   \end{table*}

\section{Conclusion}

   Our main contribution in this paper was to define a new tensor space, extending the conventional matrix space to fourth-order tensors. The key ingredient in this construction is defining a multiplication on a multidimensional discrete transforms. This new framework gives us an opportunity to design tensor products that match the physical interpretations across different modes, e.g., using a transform that captures periodicity in one mode while a new transform that reflects spatial correlations in another mode.

   We consider the SVD and QR decomposition. Although they are structurally similar to the well-known matrix counterparts, those two decompositions possess fundamental differences. Moreover, we apply this new tensor framework to both video compression and one-shot face recognition, and obtain significant performance improvements.

\bibliographystyle{IEEEbib}
\bibliography{Tensor_4D}

\end{document}